\documentclass{amsart}

\usepackage{graphicx}
\usepackage{morefloats,bm,caption,amsbsy,enumerate,amsmath,
amssymb,mathtools,amsfonts,multirow,verbatim,tikz,tikz-cd,thmtools,thm-restate}
\usepackage{enumitem}
\usepackage[alphabetic]{amsrefs}
\usepackage{chngcntr}

\usepackage[hidelinks,colorlinks=true,linkcolor=blue, citecolor=black,linktocpage=true]{hyperref}
\usetikzlibrary{matrix,arrows,decorations.pathmorphing}
\usepackage[top=1.25in, bottom=1in, left=1.25in, right=1.25in,marginpar=1in]{geometry}
\usepackage{url}
\usepackage{xcolor}
\newcommand{\red}{\textcolor{red}}

\counterwithin{figure}{section}

\newtheorem{innercustomthm}{Theorem}
\newenvironment{customthm}[1]
  {\renewcommand\theinnercustomthm{#1}\innercustomthm}
  {\endinnercustomthm}

\newtheorem{thm}{Theorem}[section]
\newtheorem{prop}[thm]{Proposition}

\newtheorem{cor}[thm]{Corollary}
\newtheorem{lem}[thm]{Lemma}

\theoremstyle{definition}
\newtheorem{define}[thm]{Definition}

\theoremstyle{remark}
\newtheorem{rem}{Remark}

\newcommand{\ve}[1]{\boldsymbol{\mathbf{#1}}}
\newcommand{\R}{\mathbb{R}}

\newcommand{\Z}{\mathbb{Z}}
\newcommand{\N}{\mathbb{N}}

\renewcommand{\d}{\partial}
\renewcommand{\subset}{\subseteq}

\renewcommand{\tilde}{\widetilde}
\renewcommand{\bar}{\overline}

\newcommand{\iso}{\cong}

\DeclareMathOperator{\br}{{br}}

\DeclareMathOperator{\Chain}{{Chain}}

\DeclareMathOperator{\cob}{{cob}}

\DeclareMathOperator{\Hom}{{Hom}}

\DeclareMathOperator{\id}{{id}}

\DeclareMathOperator{\Spin}{{Spin}}

\DeclareMathOperator{\Tor}{{Tor}}

\DeclareMathOperator{\rk}{{rk}}

\DeclareMathOperator{\Sym}{{Sym}}

\newcommand{\bF}{\mathbb{F}}

\newcommand{\bK}{\mathbb{K}}
\newcommand{\bL}{\mathbb{L}}

\newcommand{\bT}{\mathbb{T}}

\newcommand{\cA}{\mathcal{A}}

\newcommand{\cC}{\mathcal{C}}

\newcommand{\cF}{\mathcal{F}}
\newcommand{\cG}{\mathcal{G}}

\newcommand{\cM}{\mathcal{M}}

\newcommand{\cO}{\mathcal{O}}

\newcommand{\frs}{\mathfrak{s}}

\newcommand{\cHFK}{\mathcal{H\!F\! K}}
\newcommand{\cCFK}{\mathcal{C\!F\!K}}

\newcommand{\HF}{\mathit{HF}}
\newcommand{\CFK}{\mathit{CFK}}
\newcommand{\HFK}{\mathit{HFK}}

\newcommand{\CFL}{\mathit{CFL}}
\newcommand{\HFL}{\mathit{HFL}}

\newcommand{\Fus}{\cF\!\mathit{us}}
\newcommand{\cOrd}{\cO\mathit{rd}}

\newcommand{\xs}{\ve{x}}
\newcommand{\ys}{\ve{y}}
\newcommand{\zs}{\ve{z}}
\newcommand{\ws}{\ve{w}}

\newcommand{\as}{\ve{\alpha}}
\newcommand{\bs}{\ve{\beta}}

\renewcommand{\a}{\alpha}
\renewcommand{\b}{\beta}
\newcommand{\g}{\gamma}

\renewcommand{\red}{\mathrm{red}}

\DeclareMathOperator{\Ord}{{Ord}}

\usepackage{leftidx}

\title[Knot cobordisms and torsion in Floer homology]%
{Knot cobordisms, bridge index, and torsion in Floer homology}%

\author{Andr\'as Juh\'asz, Maggie Miller and Ian Zemke}

\thanks{AJ was was supported by a Royal Society Research Fellowship. MM was supported by NSF grant DGE-1656466. IZ was supported supported by NSF grant DMS-1703685. This project has received funding from the European Research Council (ERC)
under the European Union's Horizon 2020 research and innovation programme
(grant agreement No 674978).}



\begin{document}
\maketitle

\begin{abstract}
Given a connected cobordism between two knots in the 3-sphere,
our main result is an inequality involving
torsion orders of the knot Floer homology of the knots,
and the number of local maxima and the genus of the cobordism.
This has several topological applications:
The torsion order gives lower bounds on the bridge index and
the band-unlinking number of a knot, the fusion number of a ribbon knot, and the number of minima appearing in a slice disk of a knot.
It also gives a lower bound on the number of
bands appearing in a ribbon concordance between two knots.
Our bounds on the bridge index and fusion number are sharp for
$T_{p,q}$ and $T_{p,q}\# \bar{T}_{p,q}$, respectively.
We also show that the bridge index of $T_{p,q}$ is minimal within its concordance class.

The torsion order bounds a refinement of the cobordism distance on knots,
which is a metric. As a special case, we can bound the number of band moves
required to get from one knot to the other. We show
knot Floer homology also gives a lower bound on Sarkar's ribbon distance, and exhibit
examples of ribbon knots with arbitrarily large ribbon distance from the unknot.
\end{abstract}

\section{Introduction}

The slice-ribbon conjecture is one of the key open problems in knot theory.
It states that every slice knot is ribbon; i.e., admits a slice disk on which the radial function of the $4$-ball induces no local maxima. It is clear from this conjecture that
being able to bound the possible number of critical points of various indices
on surfaces bounding knots is a hard and important question.
In this paper, we use the torsion order of knot Floer homology
to give bounds on the number of critical points appearing in knot cobordisms
connecting two knots. As applications, we consider knot invariants
that can be interpreted in terms of knot cobordisms, such as the band-unlinking number of knots,
and the fusion number of ribbon knots.

If $K$ is a knot in $S^3$, we write $\HFK^-(K)$ for the minus version of knot Floer homology,
which is a finitely generated module over the polynomial ring $\bF_2[v]$.
The module $\HFK^-(K)$ decomposes non-canonically as
\[
\HFK^-(K)\iso \bF_2[v]\oplus \HFK^-_{\red}(K),
\]
where $\HFK_{\red}^-(K)$ denotes the $\bF_2[v]$-torsion submodule of $\HFK^-(K)$.
 See Section~\ref{sec:background} for background on knot Floer homology and
the link Floer TQFT, which we use in
the proofs of our main results.

If $M$ is an $\bF_2[v]$-module, we define
\[
\Ord_v(M):=\min\left\{k\in \N: v^k\cdot \Tor(M)=0\right\}\in \N\cup \{\infty\}.
\]

\begin{define}
If $K$ is a knot in $S^3$, we define its \emph{torsion order} as
\[
\Ord_v(K):=\Ord_v(\HFK^-(K)).
\]
\end{define}

The module $\HFK^-_{\red}(K)$ is annihilated by the action of $v^k$ for sufficiently large $k$, so $\Ord_v(K)$ is always finite.
Our main result is the following:

\begin{thm}\label{thm:cobordism-bound}
Let $K_0$ and $K_1$ be knots in $S^3$.
Suppose there is a connected knot cobordism $S$ from $K_0$ to $K_1$ with $M$ local maxima. Then
\[
\Ord_v(K_0)\le \max\{M, \Ord_v(K_1)\} + 2g(S).
\]
\end{thm}

One particularly notable consequence (Collary~\ref{cor:bridge-number}) of this result is the inequality
\[
\Ord_v(K)\le \br(K)-1,
\]
where $\br(K)$ is the bridge index of the knot $K$. This is the first instance in the literature of knot Floer homology producing a lower bound on the bridge index of a knot.
We now describe further topological applications of Theorem~\ref{thm:cobordism-bound}.

\subsection{Ribbon concordances }

A knot concordance with no local maxima is called a \emph{ribbon concordance}. The
notion of ribbon concordance was introduced by Gordon \cite{Gordon}.  Suppose
there is a ribbon concordance from $K_0$ to $K_1$ with $b$
saddles. One implication of Theorem~\ref{thm:cobordism-bound} is that
$\Ord_v(K_0)\le \Ord_v(K_1)$, though  this also follows from previous
work of the third author \cite{ZemRibbon}*{Theorem~1.7}. If we reverse the
roles of $K_0$ and $K_1$ in Theorem~\ref{thm:cobordism-bound}, we get that
\[
\Ord_v(K_1) \le  \max\{b, \Ord_v(K_0) \}.
\]
Hence, we obtain the following:

\begin{cor}
  Suppose that there is a ribbon concordance from $K_0$ to $K_1$ with $b$ saddles.
  Then either $b \le \Ord_v(K_0) = \Ord_v(K_1)$, or $\Ord_v(K_0) \le \Ord_v(K_1) \le b$.
\end{cor}

In particular, given knots $K_0$ and $K_1$ such that $\Ord_v(K_0) \neq \Ord_v(K_1)$,
any ribbon concordance from $K_0$ to $K_1$ must have at least $\Ord_v(K_1)$ saddles.

We can also apply Theorem \ref{thm:cobordism-bound} in the case when there is
a \emph{ribbon cobordism} $S$ of arbitrary genus from $K_0$ to $K_1$. By definition,
$S$ has no local maxima, so
\[
\Ord_v(K_0) \le \Ord_v(K_1) + 2g(S).
\]
So we obtain the following corollary:

\begin{cor}
Suppose there is a ribbon cobordism from $K_0$ to $K_1$ of genus $g$. Then
\[
\Ord_v(K_0) - \Ord_v(K_1) \le 2g.
\]
\end{cor}

\subsection{Local minima of slice disks}

Suppose $K$ is a slice knot with slice disk $D$,
and let $m$ be the number of local minima of the radial function on $B^4$ restricted to $D$.
Viewing $D$ as a cobordism from $K$ to the empty knot, it has $m$ local maxima.
By removing a ball about one of the local maxima, we obtain a concordance from $K$ to
the unknot $U$ with $m-1$ local maxima. Since $\Ord_v(U) = 0$, Theorem~\ref{thm:cobordism-bound}
implies the following:

\begin{cor}
Suppose that $D$ is a slice disk for $K$, and let $m$ denote the number of local minima of the radial function
on $B^4$ restricted to $D$. Then
\[
\Ord_v(K) \le m-1.
\]
\end{cor}

\subsection{The refined cobordism distance}

If $K_0$ and $K_1$ are knots in $S^3$, we define the \emph{refined cobordism distance} $d(K_0,K_1)$ as
the minimum of the quantity $\max\{m, M\} + 2g(S)$ over all connected, oriented knot
cobordisms $S$ from $K_0$ to $K_1$, where $m$ is the number of local minima
and $M$ is the number of local maxima of the height function on $S$. The function $d$ is a metric on the set of
knots in $S^3$ modulo isotopy; see Proposition~\ref{prop:metric}. Furthermore, $d$ is a refinement of the
standard cobordism distance on knots (i.e., the slice genus of $K_0 \# \bar{K}_1)$.
See Section~\ref{sec:metrics} for more details.
As a corollary of Theorem~\ref{thm:cobordism-bound}, we obtain the following:

\begin{cor}\label{cor:distance}
If $K_0$ and $K_1$ are knots in $S^3$, then
\[
|\Ord_v(K_0) - \Ord_v(K_1)| \le d(K_0,K_1)\le d_B(K_0,K_1),
\]
where $d_B(K_0,K_1)$ is the minimum number of oriented band moves required to get from $K_0$ to $K_1$.
\end{cor}

\begin{proof}
We first show the rightmost inequality of Corollary \ref{cor:distance}. If $b$ denotes the number of saddle points on $S$, then $2g(S) = -\chi(S) = b - m - M$. Hence
\[
\max\{m, M\} + 2g(S) = \max\{b-m, b-M\} \le b,
\]
and the distance $d(K_0, K_1)$ is at most the number of saddles appearing in any connected,
oriented cobordism from $K_0$ to $K_1$

Now we prove the leftmost inequality by utilizing Theorem~\ref{thm:cobordism-bound}. In particular, we obtain that
\begin{equation}\label{eq:distance0}
\Ord_v(K_0)\le \max \{M, \Ord_v(K_1)\} + 2g(S) \le M + \Ord_v(K_1) + 2g(S).
\end{equation}
Consequently,
\[
\Ord_v(K_0) - \Ord_v(K_1) \le M + 2g(S).
\]
Reversing the roles of $K_0$ and $K_1$ yields the statement.
\end{proof}

\subsection{The band-unlinking number}

If $K$ is a knot, the \emph{unknotting number} $u(K)$ is the minimum number
of crossing changes one must perform until one obtains the unknot. The
\emph{band-unknotting number} $u_b(K)$ is the minimum number of (oriented)
bands one must attach until one obtains an unknot. Since any crossing change
can be obtained by attaching two bands,
\[
u_b(K)\le 2 u(K).
\]

%

The band unknotting number, as well as an infinite family of variations, was
described by Hoste, Nakanishi, and Taniyama \cite{HNTH2move}, though the
concept is classical; see e.g.~Lickorish~\cite{LickorishTwistedBand}. In
their terminology, attaching an oriented band is an $\mathit{SH}(2)$-move.
They also studied the unoriented band unknotting number, which is often
called the $\mathit{H}(2)$-unknotting number.

In our present work, we are interested in a variation, which we call the \emph{band-unlinking number}, $ul_b(K)$,
which is the minimum number of oriented band moves necessary to reduce $K$ to an unlink. Note that
\[
ul_b(K)\le u_b(K).
\]
The band-unlinking and unknotting numbers are related to other topological invariants as follows:
\begin{equation}
2g_4(K)\le 2g_r(K)\le ul_b(K)\le u_b(K)\le 2g_3(K).\label{eq:relation-between-unlinking-number}
\end{equation}

In Equation~\eqref{eq:relation-between-unlinking-number}, $g_4$ is the slice
genus, $g_r$ is the ribbon slice genus (the minimal genus of a knot cobordism
from $K$ to the unknot with only saddles and local maxima), and $g_3$ is the
Seifert genus. The inequality involving the Seifert genus is obtained by
attaching bands corresponding to a basis of arcs for a minimal genus Seifert
surface.

As a corollary of Theorem~\ref{thm:cobordism-bound}, we have the following:

\begin{cor}
\label{thm:band-unknotting-number}
If $K$ is a knot in $S^3$, then
\[
\Ord_v(K)\le ul_b(K).
\]
\end{cor}

\begin{proof}
Let $b = ul_b(K)$. Then, after suitably attaching $b$ oriented bands to $K$,
we obtain an unlink of say $M$ components. By capping $M-1$ components
of the unlink, we obtain a cobordism $S$ from $K$ to the unknot $U$ with $0$ local minima,
$b$ saddles, and $M-1$ local maxima. Then
\[
2g(S) = -\chi(S) = b - M + 1,
\]
and since $\Ord_v(U) = 0$, Theorem~\ref{thm:cobordism-bound} implies that
\[
\Ord_v(K) \le \max\{M-1, \Ord_v(U)\} + 2g(S) = b,
\]
completing the proof.
\end{proof}

\begin{rem}
Corollary~\ref{thm:band-unknotting-number} and the inequality $ul_b(K)\le u_b(K)\le 2u(K)$ yield $\Ord_v(K)\le 2u(K)$. However, it is already known by 
Alishahi--Eftekhary~\cite[Theorem 1.1]{AEUnknotting} that $\Ord_v(K)\le u(K)$.
\end{rem}

\subsection{Ribbon knots and the fusion number}

A knot $K$ in $S^3$ is \emph{smoothly slice} if it bounds a smoothly embedded
disk in $B^4$. A knot $K$ is \emph{ribbon} if it bounds a smooth disk which
has only index 0 and 1 critical points with respect to the radial
function on $B^4$. Equivalently, a knot $K$ is ribbon if it can be formed by
attaching $n-1$ bands to an $n$-component unlink. The \emph{fusion number}
$\Fus(K)$ of a ribbon knot $K$ is the minimal number of bands required in any
ribbon disk for $K$; see e.g.~Miyazaki~\cite{Miyazaki86}.
Concerning the fusion number, we have the following consequence of Corollary~\ref{thm:band-unknotting-number}:

\begin{cor}\label{cor:fusion-number-bound}
If $K$ is a ribbon knot in $S^3$, then
\[
\Ord_{v}(K)\le \Fus(K).
\]
\end{cor}
\begin{proof}
If $B_1,\dots, B_b$ are the bands of a ribbon disk, then $B_1,\dots, B_b$ split $K$ into an unlink. Consequently, $ul_b(K)\le b$, so the statement follows from Corollary~\ref{thm:band-unknotting-number}.
\end{proof}

\subsection{The bridge index}

If $K$ is a knot in $S^3$, the \emph{bridge index} of $K$, denoted $\br(K)$, is the minimum over all diagrams $D$ of $K$ of the number of local maxima of $D$ with respect to a height function on the plane. It is well known that there is a ribbon disk for $K\# \bar{K}$ which has $\br(K)-1$ bands; see Figure \ref{fig:disk}. Consequently
\begin{equation}
\Fus(K\# \bar{K})\le \br(K)-1.\label{eq:fusion-bridge}
\end{equation}

 Ozsv\'{a}th and Szab\'{o}'s connected sum formula \cite{OSKnots}*{Theorem~7.1} implies
\begin{equation}
\Ord_v(K_1\# K_2)=\max \{\Ord_{v}(K_1),\Ord_v(K_2)\}.\label{eq:torsion-connected-sum}
\end{equation}
Consequently, we obtain the following additional consequence of Corollary~\ref{thm:band-unknotting-number}:
\begin{cor}\label{cor:bridge-number} If $K$ is a knot in $S^3$, then
\[
\Ord_v(K)\le \br(K)-1.
\]
\end{cor}

\begin{figure}[ht!]
	\centering
	\scalebox{.8}{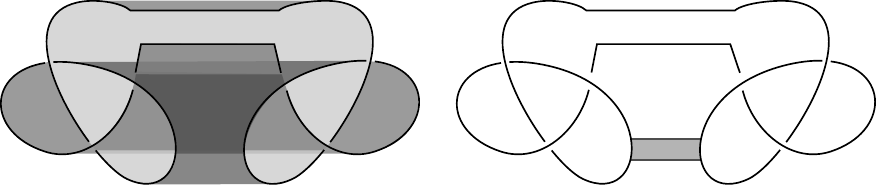}
	\caption{Left: the standard ribbon disk for $K \# \bar{K}$ (in this illustration, $K$ is a trefoil knot), immersed in $S^3$.
    Right: The corresponding $\br(K)-1$ bands attached to $K \# \bar{K}$ to obtain a $\br(K)$-component unlink.}\label{fig:disk}
\end{figure}

\subsection{Sharpness and torus knots}

As examples, we consider the positive torus knots $T_{p,q}$. It is a classical result of Schubert \cite{SchubertBridgeNumber} that
\begin{equation}
 \br(T_{p,q})=\min\{ p,q\}.\label{eq:bridge-number-torus-knots}
 \end{equation}
 Combining Equations~\eqref{eq:fusion-bridge} and~\eqref{eq:bridge-number-torus-knots}, we obtain
\begin{equation}\label{eqn:fusion}
\Fus(T_{p,q}\# \bar{T}_{p,q})\le \min\{p,q\}-1.
\end{equation}

 In Corollary~\ref{cor:torsion-order-Tpq}, we show
 \begin{equation}
\Ord_v(T_{p,q})=\min\{p,q\}-1. \label{eq:Tpq-torsion-order}
 \end{equation}
Equations~\eqref{eq:bridge-number-torus-knots} and~\eqref{eq:Tpq-torsion-order} imply
Corollaries~\ref{cor:fusion-number-bound} and~\ref{cor:bridge-number} are sharp:
\[
\Ord_{v} (T_{p,q})=\br(T_{p,q})-1 \qquad \text{and} \qquad \Ord_v (T_{p,q}\# \bar{T}_{p,q})=\Fus(T_{p,q}\# \bar{T}_{p,q}).
\]

Dai, Hom, Stoffregen, and Truong~\cite{DaiHomomorphisms} constructed a concordance invariant $N(K)$.
By \cite{DaiHomomorphisms}*{Proposition~1.15}, this satisfies
\begin{equation}
N(K)\le \Ord_v(K).\label{eq:inequality-N-torsion}
\end{equation}
In \cite{DaiHomomorphisms}*{Proposition~1.5}, they computed the invariant $N$ for L-space knots
using Ozsv\'{a}th and Szab\'{o}'s description of the knot Floer complexes of L-space knots \cite{OSLSpaceSurgeries}.
Combined with Lemma~\ref{lem:torsion-order-L-space-knot}, below, for an L-space knot $J$, we have
\begin{equation}
N(J) = \Ord_v(J). \label{L-space-torsion-order}
\end{equation}
Using equations~\eqref{eq:inequality-N-torsion} and~\eqref{L-space-torsion-order},
if $K$ is concordant to an L-space knot $J$, then
\[
\Ord_v(K) \ge \Ord_v(J).
\]

As a consequence of our bound on the bridge index in Corollary~\ref{cor:bridge-number},
together with the fact that $N(K)$ is a concordance invariant, we obtain the following:

\begin{cor}
If $K$ is concordant to a torus knot $T_{p,q}$, then
\[
\br(K)\ge \br(T_{p,q}).
\]
\end{cor}

\begin{proof}
We have
\[
\br(K)\ge \Ord_v(K)+1\ge N(K)+1=N(T_{p,q})+1=\br(T_{p,q}).
\]
The first inequality follows from Corollary~\ref{cor:bridge-number}, while
the second from Equation~\eqref{eq:inequality-N-torsion}. The first equality
holds since $N$ is a concordance invariant. The final equality follows from
Equations~\eqref{eq:bridge-number-torus-knots}, \eqref{eq:Tpq-torsion-order}, and~\eqref{L-space-torsion-order}.
\end{proof}

\subsection{Sarkar's ribbon distance}\label{sec:ribbon-distance-intro}

We first introduce the torsion distance of two knots.

\begin{define}
  Let $K$ and $K'$ be knots in $S^3$. Then we define their \emph{torsion distance} $d_t(K, K')$ as
  \[
  \min \{\, d \in \N : v^d \HFK^-(K)\iso v^d\HFK^-(K') \,\}.
  \]
\end{define}

Sarkar~\cite{SarkarRibbon} introduced the \emph{ribbon distance} $d_r(K,K')$ between knots $K$ and $K'$;
see Section~\ref{sec:ribbon-distance} for a precise definition. This is finite if and only if $K$ and $K'$ are concordant.
He proved that Lee's perturbation of Khovanov homology \cite{LeePertKh} gives a lower bound on the ribbon distance.
We prove the following knot Floer homology analogue of Sarkar's result:

\begin{thm}\label{thm:ribbon-distance}
Suppose $K$ and $K'$ are knots in $S^3$. Then
\[
d_t(K, K') \le d_r(K, K').
\]
\end{thm}

Note that $d_t(K, U) = \Ord_v(K)$, where $U$ denotes the unknot.
Hence $\Ord_v(K) \le d_r(K,U)$, and equations~\eqref{eq:torsion-connected-sum} and \eqref{eq:Tpq-torsion-order} imply that
\[
\min\{p,q\}-1 = \Ord_v(T_{p,q}\# \bar{T}_{p,q}) \le d_r(T_{p,q}\# \bar{T}_{p,q},U).
\]
On the other hand, when $K$ is ribbon, $d_r(K, U) \le \Fus(K)$.
By equation~\eqref{eqn:fusion}, we obtain that
\[
d_r(T_{p,q}\# \bar{T}_{p,q},U) = \min\{p,q\}-1.
\]
As a consequence, $d_r(K,U)$ can be arbitrarily large for ribbon knots $K$,
a fact that Sarkar was unable to establish using Khovanov homology;
see~\cite[Example~3.1]{SarkarRibbon}.

\begin{rem}
It is easy to extend this computation to show that there are prime slice knots with determinant 1
that have arbitrarily large ribbon distance from the unknot.
Kim~\cite{Kim} showed that every knot $K$ admits an invertible
concordance $C$ to a \emph{prime} knot $K'$ with the same Alexander polynomial,
obtained by taking a certain satellite of $K$. According to \cite{JMConcordance}*{Theorem~1.6}, the concordance map for $C$ (for an appropriate choice of decoration) is injective,  and hence $\Ord_v(K) \le \Ord_v(K')$.
If $K = T_{p,q} \# \bar{T}_{p,q}$ with $p$ and $q$ odd, then $\det(K) = 1$, and hence $\det(K') = 1$ as well.
\end{rem}

\subsection{Data from the knot table}\label{sec:data}

One advantage of using $\Ord_v(K)$ to bound $u_b(K)$ is computability. In
particular, a program of Ozsv\'{a}th and Szab\'{o} \cite{SzaboProgram} can
quickly compute $\HFK^-(K)$ and $\Ord_v(K)$. Using this program and data from
KnotInfo \cite{knotinfo}, we determined $\Ord_v(K)$ for all prime $K$ with
crossing number at most twelve. The results are contained in Table
\ref{table:data}. These small knots have small bridge number, so it is an
unsurprising result that all such knots have $\Ord_v(K)\in\{1,2\}$. (We
remind the reader that the unknot $U$ is not prime, and $\Ord_v(U)=0$.)

\begin{table}
\begin{centering}
\begin{tabular}{ccccccc}
\hline\hline\multicolumn{7}{c}{Knots with bridge index $3$}\\
\hline
$\mathtt{8_{19}}$&
$\mathtt{10_{124}}$&
$\mathtt{10_{128}}$&
$\mathtt{10_{139}}$&
$\mathtt{10_{152}}$&
$\mathtt{10_{154}}$&
$\mathtt{10_{161}}$\\
$\mathtt{11n_{9}}$&
$\mathtt{11n_{27}}$&
$\mathtt{11n_{57}}$&
$\mathtt{11n_{61}}$&
$\mathtt{11n_{88}}$&
$\mathtt{11n_{104}}$&
$\mathtt{11n_{126}}$\\
$\mathtt{11n_{133}}$&
$\mathtt{11n_{183}}$&
$\mathtt{12n_{0068}}$&
$\mathtt{12n_{0087}}$&
$\mathtt{12n_{0089}}$&
$\mathtt{12n_{0091}}$&
$\mathtt{12n_{0093}}$\\
$\mathtt{12n_{0105}}$&
$\mathtt{12n_{0110}}$&
$\mathtt{12n_{0129}}$&
$\mathtt{12n_{0136}}$&
$\mathtt{12n_{0138}}$&
$\mathtt{12n_{0141}}$&
$\mathtt{12n_{0149}}$\\
$\mathtt{12n_{0153}}$&
$\mathtt{12n_{0156}}$&
$\mathtt{12n_{0172}}$&
$\mathtt{12n_{0175}}$&
$\mathtt{12n_{0187}}$&
$\mathtt{12n_{0192}}$&
$\mathtt{12n_{0203}}$\\
$\mathtt{12n_{0207}}$&
$\mathtt{12n_{0217}}$&
$\mathtt{12n_{0218}}$&
$\mathtt{12n_{0228}}$&
$\mathtt{12n_{0231}}$&
$\mathtt{12n_{0242}}$&
$\mathtt{12n_{0243}}$\\
$\mathtt{12n_{0244}}$&
$\mathtt{12n_{0251}}$&
$\mathtt{12n_{0260}}$&
$\mathtt{12n_{0264}}$&
$\mathtt{12n_{0292}}$&
$\mathtt{12n_{0328}}$&
$\mathtt{12n_{0329}}$\\
$\mathtt{12n_{0366}}$&
$\mathtt{12n_{0368}}$&
$\mathtt{12n_{0374}}$&
$\mathtt{12n_{0386}}$&
$\mathtt{12n_{0387}}$&
$\mathtt{12n_{0404}}$&
$\mathtt{12n_{0417}}$\\
$\mathtt{12n_{0418}}$&
$\mathtt{12n_{0419}}$&
$\mathtt{12n_{0425}}$&
$\mathtt{12n_{0426}}$&
$\mathtt{12n_{0436}}$&
$\mathtt{12n_{0472}}$&
$\mathtt{12n_{0473}}$\\
$\mathtt{12n_{0502}}$&
$\mathtt{12n_{0503}}$&
$\mathtt{12n_{0518}}$&
$\mathtt{12n_{0528}}$&
$\mathtt{12n_{0574}}$&
$\mathtt{12n_{0575}}$&
$\mathtt{12n_{0579}}$\\
$\mathtt{12n_{0591}}$&
$\mathtt{12n_{0594}}$&
$\mathtt{12n_{0603}}$&
$\mathtt{12n_{0639}}$&
$\mathtt{12n_{0640}}$&
$\mathtt{12n_{0647}}$&
$\mathtt{12n_{0648}}$\\
$\mathtt{12n_{0655}}$&
$\mathtt{12n_{0660}}$&
$\mathtt{12n_{0665}}$&
$\mathtt{12n_{0679}}$&
$\mathtt{12n_{0680}}$&
$\mathtt{12n_{0688}}$&
$\mathtt{12n_{0689}}$\\
$\mathtt{12n_{0690}}$&
$\mathtt{12n_{0691}}$&
$\mathtt{12n_{0692}}$&
$\mathtt{12n_{0693}}$&
$\mathtt{12n_{0694}}$&
$\mathtt{12n_{0696}}$&
$\mathtt{12n_{0725}}$\\
$\mathtt{12n_{0749}}$&
$\mathtt{12n_{0750}}$&
$\mathtt{12n_{0810}}$&
$\mathtt{12n_{0830}}$&
$\mathtt{12n_{0850}}$&
$\mathtt{12n_{0851}}$&
$\mathtt{12n_{0888}}$\\ \hline \hline
\multicolumn{7}{c}{Knots with bridge index $4$}\\
\hline
$\mathtt{11n_{77}}$&
$\mathtt{11n_{81}}$&
$\mathtt{12n_{0059}}$&
$\mathtt{12n_{0067}}$&
$\mathtt{12n_{0220}}$&
$\mathtt{12n_{0229}}$&
$\mathtt{12n_{0642}}$\\
\hline
\end{tabular}
\end{centering}
\caption{These prime knots each have torsion order two in $\HFK^-$.
All other prime knots through twelve crossings have torsion order one.
Here we do not distinguish between $K$ and $\bar{K}$, as $\Ord_v(K)=\Ord_v(\bar{K})$.
Note that most of these examples have bridge index three.
When $K$ in this table has $\br(K)=3$, the bound $\Ord_v(K)\le\br(K)-1$ of Corollary \ref{cor:bridge-number} is sharp.}
\label{table:data}\end{table}

\subsection{Generalized torsion orders}
There is a larger version of the knot Floer complex, denoted $\cCFK^-(K)$,
which is a chain complex over the two-variable polynomial ring $\bF_2[u,v]$.
Since $\bF_2[u,v]$ is not a PID, the correct notion of torsion order is
somewhat subtle. For example, for many knots, $\cHFK^-(K)$ is
torsion-free over $\bF_2[u,v]$, but not free as an $\bF_2[u,v]$-module. See
Lemma~\ref{lem:generalized-torsion-order-computation} for some example
computations.

In Section~\ref{sec:generalized-torsion}, we describe several notions of
torsion order using $\cCFK^-(K)$. The largest of these we call the
\emph{chain torsion order}, denoted $\cOrd^{\Chain}_{u,v}(K)$, which is a slight
generalization of the invariant $u'(K)$ described by Alishahi and Eftekhary
\cite{AEUnknotting}. We define $\cOrd^{\Chain}_{u,v}(K)$ to be the minimal
integer $N\in \N$ such that for all $i,j\ge 0$ such that $i+j=N$, there are
chain maps
\[
f\colon \cCFK^-(K)\to \bF_2[u,v]\quad \text{and} \quad g\colon \bF_2[u,v]\to \cCFK^-(K)
\]
such that $g\circ f$ and $f\circ g$ are chain homotopic to multiplication by $u^iv^j$.

We prove that the chain torsion order satisfies a bound similar to
Theorem~\ref{thm:cobordism-bound}; see
Proposition~\ref{prop:bound-generalized-torsion}. As a consequence, we obtain
that the chain torsion order bounds the band-unlinking number $ul_b(K)$, as
well as the fusion number $\Fus(K)$ of a ribbon knot.

It is interesting to note that since $\bF_2[u,v]$ is not a PID, the behavior
of torsion under connected sums is somewhat complicated. Hence the proof of
Corollary~\ref{cor:bridge-number} does not extend to show that
$\cOrd_{u,v}^{\Chain}(K)$ is a lower bound on $\br(K)-1$. In fact,
\[
\cOrd_{u,v}^{\Chain}(T_{p,q})=(p-1)(q-1)/2,
\]
when $p$ and $q$ are positive and coprime, so such a bound cannot hold.

Nonetheless, our bound on the fusion number of a ribbon knot implies
$\cOrd_{u,v}^{\Chain}(K\# \bar{K})\le \br(K)-1$, which can be contrasted with
the fact that
\[
\cOrd_{u,v}^{\Chain}(T_{p,q}\# \bar{T}_{p,q})=\min \{p,q\}-1=\br(T_{p,q})-1,
\]
when $p$ and $q$ are positive and coprime.

\subsection{Previous bounds}\label{sec:prev-bounds}

Bounding the fusion number is challenging, though there are some bounds already in the literature.
A classical lower bound is provided by $\rk(H_1(\Sigma(K)))/2$,
where $\Sigma(K)$ is the branched double cover of $S^3$ along $K$, and $\rk$ denotes
the smallest cardinality of a generating set; see Nakanishi and Nakagawa~\cite[Proposition~2]{NakanishiNakagawa}
and Sarkar~\cite[Section~3]{SarkarRibbon}.
Following \cite[Example~3.1]{SarkarRibbon}, if $K$ is a ribbon knot with $\det(K) \neq 1$
(e.g., $K = T_{2,3} \# \bar{T}_{2,3}$), and $K_n$ is the
connected sum of $n$ copies of $K$, then $\Fus(K_n) \ge n$.
This classical method fails when $\det(K) = 1$; e.g., for $K = T_{p,q} \# \bar{T}_{p,q}$ with $p$ and $q$ odd.
Our methods allow us to show that $\Fus(K)$ can be arbitrarily large even when $\det(K) = 1$;
e.g., for $K = T_{p,q} \# \bar{T}_{p,q}$ when $p$ and $q$ are odd.

Kanenobu \cite{KanenobuBandSurgery}*{Theorem~4.3} proved a bound which
involves the dimensions of $H_1(\Sigma(K);\Z_3)$ and $H_1(\Sigma(K); \Z_5)$.
Mizuma~\cite{MizumaRibbonJones}*{Theorem~1.5} showed that  if $K$ is a ribbon knot
which has Alexander polynomial 1 and whose Jones polynomial has non-vanishing
derivative at $t=-1$, then $K$ has fusion number at least 3. More recently,
Aceto, Golla, and Lecuona \cite{AGLRational}*{Corollary~2.3}  have given obstructions using the Casson--Gordon
signature invariants of $\Sigma(K)$. Note
that these bounds do not give useful information for the ribbon knots
$K=T_{p,q}\# \bar{T}_{p,q}$ for odd $p$ and $q$ since they involve
$H_1(\Sigma(K))$, and $\Sigma(T_{p,q}\# \bar{T}_{p,q})$ is the connected sum
of the Brieskorn spheres $\Sigma(2,p,q)$ and $-\Sigma(2,p,q)$.

Alishahi \cite{AKhUnknotting} and Alishahi--Eftekhary \cite{AEUnknotting}
have obtained bounds for the unknotting number using the torsion order of
knot Floer homology and Lee's perturbation of Khovanov homology, which are
similar in flavor to our present work.

The work of Sarkar \cite{SarkarRibbon} is the most similar to ours. Sarkar
used the torsion order of the $X$-action on Lee's perturbation of Khovanov
homology to give a lower bound on the fusion number and the ribbon distance.
We note that the torsion order of Khovanov homology is usually very small.
Khovanov thin knots have torsion order at most 1. Prior to the work of
Manolescu and Marengon \cite{MMKnightMove}, the largest known torsion order
was 2. Their work exhibits a knot with torsion order at least 3. In
contrast, the $(p,q)$-torus knot has knot Floer homology with torsion order
$\min\{p,q\}-1$; see Section~\ref{sec:HFK-torus-knots}.

\subsection*{Acknowledgements}

We would like to thank Paolo Aceto and Marc Lackenby for helpful discussions.
We are also grateful to the anonymous referee for helpful comments.

\section{A refinement of the cobordism distance}
\label{sec:metrics}

Suppose that $K_0$ and $K_1$ are knots in $S^3$. The standard cobordism
distance between $K_0$ and $K_1$ is defined as the minimal genus of an
oriented knot cobordism connecting $K_0$ and $K_1$; see Baader~\cite{BaaderScissor}.
We will write $d_{\cob}(K_0,K_1)$ for the standard cobordism distance. Equivalently,
$d_{\cob}$ can be defined in terms of the slice genus of $K_0\# \bar{K}_1$.
The distance $d_{\cob}(K_0,K_1)=0$ if and only if $K_0$ and $K_1$ are
concordant, and hence descends to a metric on the knot concordance group.
In this section, we describe a refinement of the standard
cobordism distance, which is an actual metric
on the set of knots in $S^3$ modulo isotopy. Note that we will always perturb surfaces in $[0,1]\times S^3$ so that projection to the first factor is Morse.

\begin{define}
If $S$ is a connected, oriented knot cobordism in $[0,1] \times S^3$ from
$K_0$ to $K_1$ with $m$ local minima and $M$ local maxima, then we define
the quantity $|S| \in \Z_{\ge 0}$ by the formula
\[
|S|:=\max\{m, M\} + 2g(S).
\]
We define the \emph{refined cobordism distance} from $K_0$ to $K_1$ as
\[
d(K_0,K_1) := \min\left\{|S|: S
\text{ is a connected, oriented cobordism from $K_0$ to $K_1$}\right\}.
\]
\end{define}

Note that
\[
2d_{\cob}(K_0,K_1)\le d(K_0,K_1).
\]
We now show that our refined cobordism distance is indeed a metric:

\begin{prop}\label{prop:metric}
The refined cobordism distance $d(K_0,K_1)$ defines a metric on the set of knots in $S^3$ modulo isotopy.
\end{prop}

\begin{proof}
Symmetry is clear. By definition, $d(K_0, K_1) \ge 0$. Equality holds
if and only if there is a cobordism $S$ from $K_0$ to $K_1$ with $g(S) = 0$
and no local minima or maxima, and hence no saddles as $0 = \chi(S) = m - b + M = -b$;
i.e., when $K_0$ and $K_1$ are isotopic.
Finally, the triangle inequality follows from the arithmetic inequality
\[
\max \{A+A', B+B'\}\le \max\{A,B\}+\max\{A',B'\}.
\]
\end{proof}

There is another metric on the set of knots which commonly appears in the literature, the
\emph{Gordian metric} $d_G$, introduced by Murakami~\cite{MurakamiMetrics}.
The quantity $d_G(K_0,K_1)$ is the minimal number of crossing changes
required to change $K_0$ into $K_1$. Since a crossing change may be realized
with two oriented band surgeries, we have
\[
d(K_0,K_1)\le d_B(K_0,K_1)\le 2d_G(K_0,K_1).
\]

\section{Background on knot and link Floer homologies}\label{sec:background}

\subsection{The link Floer homology groups}

Knot Floer homology is an invariant of knots in 3-manifolds defined by
Ozsv\'{a}th and Szab\'{o} \cite{OSKnots}, and independently Rasmussen
\cite{RasmussenKnots}. The construction was extended to links by Ozsv\'{a}th
and Szab\'{o} \cite{OSLinks}.

A \emph{multi-based link} $\bL=(L,\ws,\zs)$ consists of an oriented link $L$,
equipped with two disjoint collections of basepoints, $\ws$ and $\zs$, satisfying the following:
\begin{enumerate}
\item $\ws$ and $\zs$ alternate as one traverses $L$.
\item Each component of $L$ has at least 2 basepoints.
\end{enumerate}

To a multi-based link $\bL$ in $S^3$, Ozsv\'{a}th and Szab\'{o} associate
several versions of the link Floer homology groups. The hat version is a
bigraded $\bF_2$-vector space $\widehat{\HFL}(\bL)$.  We will mostly
focus on the minus version, denoted $\HFL^-(\bL)$,
which is a module over the polynomial ring $\bF_2[v]$.

The link Floer groups are constructed by picking a Heegaard diagram
$(\Sigma,\as,\bs,\ws,\zs)$ for $\bL$. Write $\as=(\alpha_1,\dots, \alpha_n)$
and $\bs=(\beta_1,\dots, \beta_n)$ for the attaching curves, and consider the
two half-dimensional tori
\[
\bT_{\a}:=\alpha_1\times \cdots \times  \alpha_n\qquad \text{and}\qquad \bT_{\b}:=\beta_1\times \cdots \times \beta_n
\]
in $\Sym^n(\Sigma)$. The module $\widehat{\CFL}(\bL)$ is defined to be
the free $\bF_2$-module generated by the intersection points $\bT_{\a}\cap \bT_{\b}$.
The module $\CFL^-(\bL)$ is the free $\bF_2[v]$-module generated
by $\bT_{\a}\cap \bT_{\b}$. The differential $\widehat{\d}$ on
$\widehat{\CFL}(\bL)$ counts rigid pseudo-holomorphic disks in $\Sym^n(\Sigma)$ with
multiplicity zero on $\ws\cup \zs$. The differential on $\CFL^-(\bL)$ is
given by
\begin{equation}
\d^-\xs=\sum_{\ys\in \bT_\a\cap \bT_{\b}}\sum_{\substack{\phi\in \pi_2(\xs,\ys)\\ \mu(\phi)=1\\ n_{\ws}(\phi)=0}} \# (\cM(\phi)/\R)  v^{n_{\zs}(\phi)}\cdot \ys, \label{eq:differential-minus}
\end{equation}
extended equivariantly over $\bF_2[v]$. The modules $\widehat{\HFL}(\bL)$ and $\HFL^-(\bL)$
are the homologies of $\widehat{\CFL}(\bL)$ and $\CFL^-(\bL)$, respectively.

The module $\HFL^-(\bL)$ decomposes (non-canonically) as
\[
\HFL^-(\bL)\iso \left( \bigoplus_{i=1}^{2^{k-1}} \bF_2[v]\right)\oplus \HFL_{\red}^-(\bL),
\]
where $k=|\ws|=|\zs|$ and $\HFL^-_{\red}(\bL)$ denotes the $\bF_2[v]$-torsion
submodule of $\HFL^-(\bL)$. Since $\HFL^-(\bL)$ admits a relative
$\Z$-grading where $v$ has grading $+1$ (the Alexander grading), the module
$\HFL^-_{\red}(\bL)$ is always isomorphic to a direct sum of modules of the form
$\bF_2[v]/(v^i)$ for $i \ge 0$. In particular, $v^l$ annihilates $\HFL^-_{\red}(\bL)$ for
all sufficiently large $l \in \N$, and hence $\Ord_v(K)$ is always finite.

There is a symmetric version of knot Floer homology that commonly appears in the literature. It is freely generated over $\bF[u]$ by intersection points $\xs\in \bT_{\a}\cap \bT_{\b}$, and its differential counts disks with $n_{\zs}(\phi)=0$, which are weighted by $u^{n_{\ws}(\phi)}$. In this setting, the variable $u$ has Maslov index $-2$, and Alexander grading $-1$.


If $\bK=(K,w,z)$ is a doubly-based knot, then, by definition, the link Floer
homology groups coincide with the knot Floer homology groups; i.e.,
$\HFK^-(\bK)\iso \HFL^-(\bK)$. Following standard notation, we will usually
write $\HFK^-(K)$ instead of $\HFK^-(\bK)$.

Ozsv\'{a}th and Szab\'{o}'s connected sum formula \cite{OSKnots}*{Theorem~7.1} implies that
\[
\CFK^-(K_1\# K_2)\iso \CFK^-(K_1)\otimes_{\bF_2[v]} \CFK^-(K_2).
\]
Consequently, by the K\"{u}nneth theorem for chain complexes over $\bF_2[v]$, we have
\begin{equation}
\Ord_{v}(K_1\# K_2)=\max\left\{\Ord_v(K_1), \Ord_v (K_2)\right\}.
\label{eq:connected-sum-Ord}
\end{equation}

Ozsv\'{a}th and Szab\'{o} also proved that the mirror of a knot has dual knot Floer homology:
\[
\CFK^-(\bar{K})\iso \Hom_{\bF_2[v]}(\CFK^-(K), \bF_2[v]).
\]
(The proof is the same as for the closed 3-manifold invariants; see Ozsv\'ath and Szab\'o \cite{OSTriangles}*{Section~5.1}).
Consequently,
\begin{equation}
\Ord_{v}(\bar{K}) = \Ord_v(K).
\label{eq:Ord-mirror}
\end{equation}
Combining equations~\eqref{eq:connected-sum-Ord} and~\eqref{eq:Ord-mirror}, we obtain that
\begin{equation}\label{eqn:sum}
\Ord_v(K \# \bar{K}) = \Ord_v(K),
\end{equation}
a result that we will use repeatedly.

\subsection{The link Floer TQFT}

We will be interested in the functorial aspects of link Floer homology.
A \emph{decorated link cobordism} between two multi-based links $\bL_0=(L_0,\ws_0,\zs_0)$ and $\bL_1=(L_1,\ws_1,\zs_1)$ is a pair $\cF=(S,\cA)$, as follows:
\begin{enumerate}
\item  $S$ is a smooth, properly embedded, oriented surface in $[0,1]\times S^3$ such that
\[
\d S=(-\{0\}\times L_0)\cup (\{1\}\times L_1).
\]
\item $\cA\subset S$ is a finite collection of properly embedded arcs, such that $S\setminus \cA$ consists of two disjoint subsurfaces, $S_{\ws}$ and $S_{\zs}$. Further, $\ws\subset S_{\ws}$ and $\zs\subset S_{\zs}$.
\end{enumerate}
Figure~\ref{fig::2} shows some examples of decorated link cobordisms.

For a decorated link cobordism $\cF$ from $\bL_0$ to $\bL_1$, there are cobordism maps
\[
\widehat{F}_{\cF}\colon \widehat{\HFL}(\bL_0)\to \widehat{\HFL}(\bL_1)\qquad \text{and} \qquad F_{\cF}\colon \HFL^-(\bL_0)\to \HFL^-(\bL_1).
\]
The construction of the map $\widehat{F}_{\cF}$ is due to the first author
\cite{JCob}, using the contact gluing map of Honda, Kazez, and Mati\'{c}~\cite{HKMTQFT}.
The third author~\cite{ZemCFLTQFT} subsequently gave an
alternate construction which also works on the minus version. Their
equivalence on the hat version was proven by the first and third authors
\cite{JZContactHandles}*{Theorem~1.4}.

The link cobordism maps satisfy a simple relation with respect to adding tubes:

 \begin{lem}\label{lem:stabilization}
 Suppose that $\cF=(S,\cA)$ is a decorated link cobordism from $\bL_0$ to
 $\bL_1$. Suppose that $\cF'$ is a decorated link cobordism obtained by
 adding a tube to $\cF$, with both feet in the $S_{\zs}$ subregion of $S$;
 see Figure~\ref{fig::2}. Then
 \[
F_{\cF'}=v\cdot F_{\cF}.
 \]
 \end{lem}

A proof of Lemma~\ref{lem:stabilization} can be found in \cite{JZStabilization}*{Lemma~5.3}. We note that if we add a tube with feet in $S_{\ws}$, then the induced map is zero on $\HFL^-$. More generally, in Section~\ref{sec:generalized-torsion}, we consider a version of knot Floer homology over the 2-variable polynomial ring $\bF[u,v]$. In this setting, adding a tube to $S_{\zs}$ has the effect of multiplication by $v$, while adding a tube to $S_{\ws}$ has the effect of multiplication by $u$.

\begin{figure}[ht!]
	\centering
\begingroup%
  \makeatletter%
  \providecommand\color[2][]{%
    \errmessage{(Inkscape) Color is used for the text in Inkscape, but the package 'color.sty' is not loaded}%
    \renewcommand\color[2][]{}%
  }%
  \providecommand\transparent[1]{%
    \errmessage{(Inkscape) Transparency is used (non-zero) for the text in Inkscape, but the package 'transparent.sty' is not loaded}%
    \renewcommand\transparent[1]{}%
  }%
  \providecommand\rotatebox[2]{#2}%
  \newcommand*\fsize{\dimexpr\f@size pt\relax}%
  \newcommand*\lineheight[1]{\fontsize{\fsize}{#1\fsize}\selectfont}%
  \ifx\svgwidth\undefined%
    \setlength{\unitlength}{278.34231062bp}%
    \ifx\svgscale\undefined%
      \relax%
    \else%
      \setlength{\unitlength}{\unitlength * \real{\svgscale}}%
    \fi%
  \else%
    \setlength{\unitlength}{\svgwidth}%
  \fi%
  \global\let\svgwidth\undefined%
  \global\let\svgscale\undefined%
  \makeatother%
  \begin{picture}(1,0.5006601)%
    \lineheight{1}%
    \setlength\tabcolsep{0pt}%
    \put(0,0){\includegraphics[width=\unitlength,page=1]{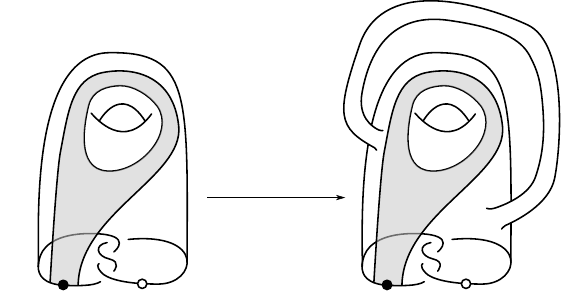}}%
    \put(0.03949863,0.12824683){\color[rgb]{0,0,0}\makebox(0,0)[rt]{\lineheight{1.25}\smash{\begin{tabular}[t]{r}$\cF$\end{tabular}}}}%
    \put(0.9565508,0.12824683){\color[rgb]{0,0,0}\makebox(0,0)[lt]{\lineheight{1.25}\smash{\begin{tabular}[t]{l}$\cF'$\end{tabular}}}}%
  \end{picture}%
\endgroup%

	\caption{Stabilizing a surface in the $\zs$-subregion.}\label{fig::2}
\end{figure}

\section{Knot Floer homology and the cobordism distance}

We begin with the main technical result needed for Theorem~\ref{thm:cobordism-bound}:

\begin{prop}\label{prop:main-thm-maps}
Suppose that $S$ is a connected, oriented knot cobordism from $K_0$ to $K_1$ in $[0,1] \times S^3$
with $m$ local minima, $b$ saddles, and $M$ local maxima, and suppose that $\cF=(S,\cA)$ is a
decoration of $S$ such that the type-$\ws$ region is a regular neighborhood of
an arc running from $K_0$ to $K_1$. Let
$\bar{\cF}$ denote the cobordism from $K_1$ to $K_0$ obtained by horizontally mirroring $\cF$. Then
\[
v^M\cdot F_{\bar{\cF}}\circ F_{\cF}=v^{b-m}\cdot \id_{\HFK^-(K_0)}.
\]
\end{prop}

\begin{proof}
We can rearrange the critical points of $S$ so that $S$ has a movie of the following form:
\begin{enumerate}[leftmargin=2cm, label=($M$-\arabic*), ref=$M$-\arabic*]
\item\label{movie:1} $m$ births, which add $m$ unknots $U_1,\dots, U_m$.
\item\label{movie:2} $m$ fusion saddles $B_1,\dots, B_m$, which merge $U_1,\dots,U_m$ with $K_0$.
\item\label{movie:3} $b-m$ additional saddles, along bands $B_{m+1},\dots, B_b$.
\item\label{movie:4} $M$ deaths, corresponding to deleting unknots $U_1',\dots, U_M'$.
\end{enumerate}
We can give a movie with 8 steps for $\bar{\cF}\circ \cF$ by first playing
\eqref{movie:1}--\eqref{movie:4} forward, and then playing them backward, in
reverse order. The fourth step of this 8-step movie is to delete the unknots
$U_1',\dots, U_M'$ via $M$ deaths. The fifth step is to add them back with
$M$ births. Consider the cobordism $\cG$ obtained by deleting these two
levels. The cobordism $\cG$ is obtained by attaching $M$ tubes to
$\bar{\cF}\circ \cF$. Since the decoration of $\cG$ is such that the
type-$\ws$ region is a neighborhood of an appropriate arc from the incoming
$K_0$ to the outgoing $K_0$, the cobordism $\cG$ is obtained by
attaching $M$ tubes to the $\zs$-subregion of $\bar{\cF}\circ \cF$.
Consequently, Lemma~\ref{lem:stabilization} implies that
\begin{equation}
F_{\cG}=v^M\cdot F_{\bar{\cF}\circ \cF}.\label{eq:FG=vmFF}
\end{equation}

The cobordism $\cG$ has the movie obtained by playing \eqref{movie:1},
\eqref{movie:2}, and \eqref{movie:3} forward, and then playing them backward,
in reverse order. The third and fourth steps of this movie describe $b-m$
tubes, added to a cobordism $\cG'$, which is obtained by first playing
\eqref{movie:1} and \eqref{movie:2}, and then playing them backwards, in
reversed order. By
Lemma~\ref{lem:stabilization}, we obtain
\begin{equation}
F_{\cG}=v^{b-m}\cdot F_{\cG'}.\label{eq:FG=vb-mG'}
\end{equation}

Finally, $\cG'$ is obtained by playing~\eqref{movie:1} and \eqref{movie:2},
and then playing them backwards, in reverse order. The births and deaths
determine 2-spheres $S_1,\dots, S_m$, and the bands and their reverses
determine tubes. Hence $\cG'$ is the cobordism obtained by tubing in the
spheres $S_1,\dots, S_m$ to the identity concordance $[0,1]\times K_0$. The
proof of \cite{ZemRibbon}*{Theorem~1.7} implies immediately that tubing in
spheres in this manner does not affect the cobordism map, so
\begin{equation}
F_{\cG'}=\id_{\HFK^-(K_0)}.\label{eq:G'=id}
\end{equation}
Combining Equations~\eqref{eq:FG=vmFF}, \eqref{eq:FG=vb-mG'},
and \eqref{eq:G'=id} yields the statement.
\end{proof}

Our main theorem is now an algebraic consequence of Proposition~\ref{prop:main-thm-maps}:

\begin{customthm}{\ref{thm:cobordism-bound}}
Suppose there is an oriented knot cobordism $S$ from $K_0$ to $K_1$ with $M$ local maxima. Then
\[
\Ord_v(K_0) \le \max\{M ,\Ord_v(K_1)\} + 2g(S).
\]
\end{customthm}

\begin{proof}
Let $\cF$ denote the cobordism obtained by decorating $S$ such that the
$\ws$-subregion is a regular neighborhood of an arc running from $K_0$ to $K_1$.
Let $\bar{\cF}$ denote the cobordism from $K_1$ to $K_0$ obtained by
horizontally mirroring $\cF$. Proposition~\ref{prop:main-thm-maps} implies that
\begin{equation}
v^M\cdot F_{\bar{\cF}\circ \cF} = v^{b-m}\cdot \id_{\HFK^-(K_0)},
\label{eq:restate-main-prop}
\end{equation}
where $m$ is the number of local minima and $b$ is the number of saddles on $S$.

Since
\[
F_{\bar{\cF}\circ \cF}=F_{\bar{\cF}}\circ F_{\cF}
\]
by the composition law, it follows that, if $x\in \HFK^-_{\red}(K_0)$,
then $F_{\bar{\cF}\circ \cF}(v^j \cdot x)=0$ if $j \ge \Ord_v(K_1)$.
On the other hand, equation~\eqref{eq:restate-main-prop} implies that
\[
F_{\bar{\cF}\circ \cF}(v^{l+M}\cdot x)=v^{b-m+l}\cdot x
\]
for all $l \ge 0$. Consequently, if $l \ge \max\{ 0,\Ord_v(K_1)-M\}$,
then $v^{b-m+l}\cdot x=0$. It follows that
\[
\Ord_{v}(K_0)\le b-m+\max\{0,\Ord_v(K_1)-M\} = \max\{M ,\Ord_v(K_1)\} + 2g(S),
\]
since $2g(S) = -\chi(S) = b - m - M$.
\end{proof}

\section{Torus knots and sharpness}
\label{sec:HFK-torus-knots}
An \emph{L-space} is a rational homology 3-sphere $Y$ such that
$\widehat{\HF}(Y,\frs)\iso \bF_2$ for each $\frs\in \Spin^c(Y)$ (this is the
smallest possible rank for a rational homology sphere). Lens spaces are
examples of L-spaces. An \emph{L-space knot} is a knot $K$ in $S^3$ such that
$S^3_p(K)$ is an L-space for some $p\in \Z$. If $p$, $q > 0$ are coprime, the
torus knot $T_{p,q}$ is an L-space knot since $pq\pm 1$ surgery on $T_{p,q}$
is the lens space $L(pq\pm 1, q^2)$.

Ozsv\'{a}th and Szab\'{o} \cite{OSLSpaceSurgeries}*{Theorem~1.2} proved that
the knot Floer homology of an L-space knot is completely determined by its
Alexander polynomial. Furthermore, they showed
\cite{OSLSpaceSurgeries}*{Corollary~1.3} that the Alexander polynomial of an
L-space knot can be written as
\[
\Delta_K(t)=\sum_{k=0}^{2n} (-1)^{k} t^{\alpha_k}
\]
for a decreasing sequence of integers $\alpha_0,\dots, \alpha_{2n}$.
Their computation implies the following:

\begin{lem}\label{lem:torsion-order-L-space-knot}
If $K$ is an $L$-space knot, and $\alpha_0,\dots, \alpha_{2n}$ are the
non-zero degrees appearing in the Alexander polynomial of $K$, written in
decreasing order, then
\[
\Ord_v(K)=\max \{\alpha_{i-1}-\alpha_{i}:1\le i\le 2n\}.
\]
\end{lem}

\begin{proof}
Mirror $K$ if necessary so that large positive surgeries on $K$ yield $L$-spaces. (As we have defined $L$-space knots, it might be that originally large negative surgeries on $K$ yield $L$-spaces.)
We first describe Ozsv\'{a}th and Szab\'{o}'s computation of
$\CFK^\infty(K)$. Note that Ozsv\'{a}th and Szab\'{o} only stated their
computation for $\widehat{\HFK}(K)$, though their proof works for
$\CFK^\infty(K)$; see \cite{OSSUpsilon}*{Theorem~2.10}.
Let $d_1,\dots, d_{2n}$ denote the gaps between the integers $\alpha_0,\dots, \alpha_{2n}$; i.e.,
\begin{equation}\label{eq:gaps-alexander}
d_{i}:=\alpha_{i-1}-\alpha_{i}.
\end{equation}
Ozsv\'{a}th and Szab\'{o} proved that $\CFK^\infty(K)$ is chain homotopy
equivalent to the staircase complex with generators $x_0,\dots, x_{2n}$ over
$\bF_2[U,U^{-1}]$, with the following differential:
\[
\d x_{2k}=0 \qquad \text{and} \qquad \d x_{2k+1}=x_{2k}+x_{2k+2}.
\]
Up to an overall shift, the $\Z\oplus \Z$-filtration  is determined by the following:
\begin{itemize}
\item The element $x_{2k}$ has the same $j$-filtration as $x_{2k+1}$,
   but the $i$-filtration differs by $d_{2k+1}$.
\item  The element $x_{2k+2}$ has the same $i$-filtration as $x_{2k+1}$,
   but the $j$-filtration differs by $d_{2k+2}$.
\end{itemize}
See Figure~\ref{fig::3} for a schematic of the staircase complex, as well as an example.

\begin{figure}[ht!]
	\centering
	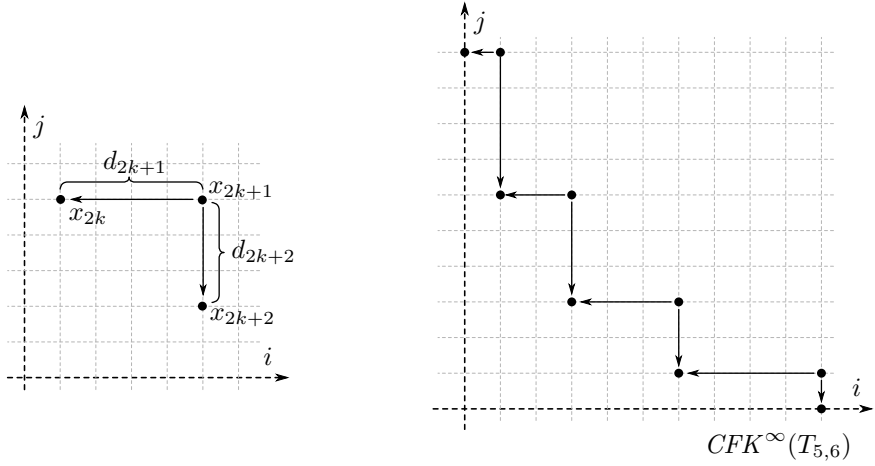
	\caption{The generators of $\CFK^\infty(K)$ on the left, for an L-space knot $K$.
     On the right is $\CFK^\infty(T_{5,6})$. The symmetrized Alexander polynomial of $T_{5,6}$ is
     $\Delta_{T_{5,6}}(t)=t^{10}-t^9+t^5-t^3+1-t^{-3}+t^{-5}-t^{-9}+t^{-10}$.}
\label{fig::3}
\end{figure}

The minus version $\CFK^-(K)$ can be read off from the above description of $\CFK^\infty(K)$, as follows:
There is one generator $y_i$ over $\bF_2[v]$ for each $x_i$. The differential satisfies
\[
\d^- y_{2k}=0\qquad \text{and} \qquad\d^-y_{2k+1}=v^{d_{2k+2}}\cdot y_{2k+2}.
\]
Consequently, when $K$ is an L-space knot, $\Ord_v(K) = \max \{ d_{2k+2}:0\le k\le n-1\}$.
Since the Alexander polynomial is symmetric, we have $d_{2k+1}=d_{2n-2k}$, so
\[
\Ord_v(K)=\max \{ d_i:1\le i\le 2n\},
\]
as claimed.
\end{proof}

We now need an elementary result concerning the Alexander polynomial of torus knots:

\begin{lem}\label{lem:Alex-poly-torus} If $p$ and $q$ are coprime, positive integers,
then the first three terms of the symmetrized Alexander polynomial of $T_{p,q}$ are
\[
\Delta_{T_{p,q}}(t)=t^d-t^{d-1}+t^{d-\min\{p,q\}}+\cdots,
\]
where $d=\frac{(p-1)(q-1)}{2}$.
\end{lem}

\begin{proof} Write
\begin{equation}
\Delta_{T_{p,q}}(t)=t^{-d}\frac{(t^{pq}-1)(t-1)}{(t^p-1)(t^q-1)}.\label{eq:alex-poly-torus}
\end{equation}
Canceling factors of $t-1$ in Equation~\eqref{eq:alex-poly-torus} and rearranging, we obtain
\begin{equation}
t^d(t^{p-1}+\cdots +1)(t^{q-1}+\cdots +1)\Delta_{T_{p,q}}(t)=t^{pq-1}+t^{pq-2}+\cdots +1.
\label{eq:alex-poly-torus2}
\end{equation}
It is a straightforward algebraic exercise to see that
 Equation~\eqref{eq:alex-poly-torus2} implies that the first three terms of $\Delta_{T_{p,q}}(t)$ are as claimed.
\end{proof}

We are now ready to show that our bounds in Corollaries~\ref{cor:fusion-number-bound} and~\ref{cor:bridge-number}
on the fusion number and the bridge index in terms of the torsion order are sharp:

\begin{cor}\label{cor:torsion-order-Tpq}
Let $T_{p,q}$ be a torus knot with $q > 0$. Then
\[
\Ord_v(T_{p,q}) = \Ord_v(T_{p,q} \# \bar{T}_{p,q}) = \Fus(T_{p,q}\# \bar{T}_{p,q})=\br(T_{p,q})-1 = \min\{|p|,q\} - 1.
\]
\end{cor}

\begin{proof}
All of the stated quantities agree for $T_{p,q}$ and $T_{-p,q}$, so, without loss of generality, take $p>0$. Combining Lemmas~\ref{lem:torsion-order-L-space-knot} and~\ref{lem:Alex-poly-torus}, we obtain that
\[
\Ord_v(T_{p,q})\ge \min\{p,q\} - 1.
\]
On the other hand, $T_{p,q} \# \bar{T}_{p,q}$ is ribbon,
and hence equation~\eqref{eqn:sum} and Corollary~\ref{cor:fusion-number-bound} imply that
\[
\Ord_v(T_{p,q}) = \Ord_v(T_{p,q} \# \bar{T}_{p,q}) \le \Fus(T_{p,q} \# \bar{T}_{p,q}).
\]
By equations~\eqref{eq:fusion-bridge} and~\eqref{eq:bridge-number-torus-knots}, we have
\[
\Fus(T_{p,q} \# \bar{T}_{p,q}) \le \br(T_{p,q}) - 1 = \min\{p,q\} - 1,
\]
and the result follows.
\end{proof}

Note that Corollaries~\ref{thm:band-unknotting-number} and \ref{cor:torsion-order-Tpq} imply that
\begin{equation}
ul_b(T_{p,q})\ge \Ord_v(T_{p,q})=\min \{p,q\}-1.\label{eq:band-unlinking-band-Tpq}
\end{equation}
However, Equation~\eqref{eq:relation-between-unlinking-number} and the fact that
\[
g_3(T_{p,q})=g_4(T_{p,q})=(p-1)(q-1)/2
\]
imply that
\[
ul_b(T_{p,q})=(p-1)(q-1),
\]
so Equation~\eqref{eq:band-unlinking-band-Tpq} is not a particularly good bound in this case.

\section{Sarkar's ribbon distance and knot Floer homology}
\label{sec:ribbon-distance}
Following Sarkar \cite{SarkarRibbon}, if $K$ and $K'$ are concordant
knots, then the \emph{ribbon distance} $d_r(K,K')$ is the minimal $k$ such
that there is a sequence of knots $K=K_0,K_1,\dots, K_n=K'$ such that there
exists a ribbon concordance connecting $K_i$ and $K_{i+1}$ (in either
direction) with at most $k$ saddles. If $K$ and $K'$ are not concordant, we
set $d_r(K,K')=\infty$.
The ribbon distance satisfies the following properties:
\begin{enumerate}
\item $d_r(K,K')<\infty$ if and only if $K$ and $K'$ are concordant.
\item $d_r(K,K')=0$ if and only if $K$ and $K'$ are isotopic.
\item $d_r(K,K')=d_r(K',K)$.
\item $d_r(K,K'')\le \max\{ d_r(K,K'), d_r(K',K'')\}.$
\end{enumerate}

Furthermore, if $K$ is ribbon, then $d_r(K,U)\le \Fus(K)$.
Inspired by \cite{SarkarRibbon}*{Theorem~1.1}, we prove the following,
which is equivalent to the statement in Section~\ref{sec:ribbon-distance-intro}:

\begin{customthm}{\ref{thm:ribbon-distance}}
Suppose $K$ and $K'$ are concordant knots, and let $d=d_r(K,K')$ denote their ribbon distance. Then
\[
v^d \HFK^-(K)\iso v^d\HFK^-(K').
\]
\end{customthm}

\begin{proof}
Since ribbon distance is defined by taking a sequence of ribbon concordances,
it is sufficient to show that if there is a single ribbon concordance $C$ from $K$ to $K'$ with $n$ saddles, then
\begin{equation}
v^n \HFK^-(K)\iso v^n \HFK^-(K'). \label{eq:main-claim-ribbon-distance}
\end{equation}
To prove Equation~\eqref{eq:main-claim-ribbon-distance}, we exhibit maps
\[
F\colon v^n \HFK^-(K)\to v^n \HFK^-(K')\qquad \text{and}\qquad G\colon v^n \HFK^-(K')\to v^n \HFK^-(K),
\]
and show that
\[
F\circ G=\id_{v^n\HFK^-(K')}\qquad \text{and}\qquad G\circ F=\id_{v^n\HFK^-(K)}.
\]

Let $\bar{C}$ be the concordance from $K'$ to $K$ obtained by horizontally
mirroring $C$. We write $\cC$ for a decoration of $C$ with two parallel
dividing arcs, and $\bar{\cC}$ for the mirrored decoration on $\bar{C}$.
Let
\[
F_0\colon \HFK^-(K)\to \HFK^-(K') \qquad \text{and} \qquad G_0\colon \HFK^-(K') \to \HFK^-(K)
\]
denote the maps induced by $\cC$ and $\bar{\cC}$, respectively.
Since $F_0$ and $G_0$ are $\bF_2[v]$-equivariant, we define $F$ and $G$ to be
the restrictions of $F_0$ and $G_0$ to the images of $v^n$.
A first application of Proposition~\ref{prop:main-thm-maps} implies that
$G_0\circ F_0 = \id_{\HFK^-(K)}$, so we easily obtain
$G \circ F = \id_{v^n\HFK^-(K)}$.

Next, Proposition~\ref{prop:main-thm-maps} implies that
\[
v^n\cdot (F_0\circ G_0)=v^n\cdot \id_{\HFK^-(K')}.
\]
Hence $(F_0 \circ G_0)(v^n\cdot x) = v^n\cdot x$; i.e.,
$F \circ G = \id_{v^n\HFK^-(K')}$, completing the proof.
\end{proof}

\section{Generalized torsion orders}\label{sec:generalized-torsion}

In this section, we describe some algebraic generalizations of the torsion
order of  $\HFK^-(K)$. We consider the full knot Floer complex $\cCFK^-(K)$,
which is a free and finitely generated chain complex over the two-variable
polynomial ring $\bF_2[u,v]$.   As an  $\bF_2[u,v]$-module, $\cCFK^-(K)$ is
freely generated by intersection points $\xs\in \bT_{\a}\cap \bT_{\b}$.
Analogous to Equation~\eqref{eq:differential-minus}, the full differential is
given by
\[
\d \xs=\sum_{\ys\in \bT_{\a}\cap \bT_{\b}} \sum_{\substack{\phi\in \pi_2(\xs,\ys)\\ \mu(\phi)=1}} \# (\cM(\phi)/\R) u^{n_w(\phi)} v^{n_z(\phi)} \cdot \ys.
\]
Write $\cHFK^-(K)$ for the homology of $\cCFK^-(K)$. Note that
\begin{equation}
\CFK^-(K)\iso \cCFK^-(K)\otimes_{\bF_2[u,v]} \bF_2[u,v]/(u).\label{eq:CFK=cCFKtensor}
\end{equation}

It is not hard to see that the torsion submodule of $\cHFK^-(K)$ is finitely
generated over $\bF_2$. Furthermore, both $u^N$ and $v^N$ annihilate the
torsion submodule of $\cHFK^-(K)$ for large $N$.
It is important to note that $\bF_2[u,v]$ is not a PID, so a finitely generated module may be
torsion-free but not free (see Figure~\ref{fig:6} for an example).

The quantities $\Ord_u(\cHFK^-(K))$ and $\Ord_v(\cHFK^-(K))$ are both well defined, non-negative integers.
The conjugation symmetry of knot Floer homology implies
\[
\Ord_u(\cHFK^-(K))=\Ord_v(\cHFK^-(K)).
\]
To distinguish between the torsion orders of $\HFK^-(K)$ and $\cHFK^-(K)$, we will write
\[
\cOrd_v(K):=\Ord_v(\cHFK^-(K)).
\]

\begin{define}\label{def:generalized-torsion-order}
We define the following additional notions of torsion order:
\begin{enumerate}
\item The \emph{2-variable torsion order} $\cOrd_{u,v}(K)$ is the smallest $N \in \N$ such that
\[
u^iv^j \cdot \Tor (\cHFK^-(K))=\{0\}
\]
whenever $i$, $j \ge 0$ and $i + j = N$.

\item The \emph{homomorphism torsion order} $\cOrd_{u,v}^{\Hom}(K)$ is  the minimal $N \in \N$
such that, whenever $i$, $j \ge 0$ satisfy $i + j = N$, there are homogeneously graded maps
\[
f\colon \cHFK^-(K) \to \bF_2[u,v]\qquad \text{and} \qquad g\colon \bF_2[u,v]\to  \cHFK^-(K)
\]
such that $f \circ g$ and $g \circ f$ are both multiplication by $u^i v^j$.
\item The \emph{chain torsion order} $\cOrd_{u,v}^{\Chain}(K)$ is the minimal $N\in \N$
such that, whenever $i$, $j \ge 0$ satisfy $i + j = N$, there are homogeneously graded chain maps
\[
f \colon \cCFK^-(K)\to \bF_2[u,v] \qquad \text{and} \qquad g \colon \bF_2[u,v] \to \cCFK^-(K)
\]
such that $f \circ g$ and $g \circ f$ are chain homotopic to multiplication by $u^iv^j$.
\end{enumerate}
\end{define}

The homomorphism and chain torsion orders  are both modifications
of the invariant $u'(K)$ described by Alishahi and Eftekhary~\cite{AETangles, AEUnknotting}.

We also clarify the meaning of a \emph{homogeneously graded} map in
Definition~\ref{def:generalized-torsion-order}: If $V$ and $W$ are graded
vector spaces, a homogeneously graded map $f\colon V\to W$ is one which
changes grading by a fixed degree. This coincides with the notion obtained by
viewing $\Hom(V,W)$ itself as a graded vector space.

A straightforward algebraic argument shows that
\begin{equation}
\cOrd_u(K)\le \cOrd_{u,v}(K)\le \cOrd_{u,v}^{\Hom}(K)\le \cOrd_{u,v}^{\Chain}(K). \label{eq:inequalities-of-torsion}
\end{equation}
The chain torsion order also has the advantage that it respects duality:
\begin{equation}
\cOrd_{u,v}^{\Chain}(K)= \cOrd_{u,v}^{\Chain}(\bar{K}). \label{eq:mirrow-symmetry}
\end{equation}
The analog of equation~\eqref{eq:mirrow-symmetry} fails for the 2-variable torsion order $\cOrd_{u,v}(K)$: In Lemma~\ref{lem:generalized-torsion-order-computation}, we show that $\cOrd_{u,v}(T_{p,q})\neq\cOrd_{u,v}(\bar{T}_{p,q})$.

\subsection{A generalized doubling relation}

We now prove the following generalization of Proposition~\ref{prop:main-thm-maps}.

\begin{prop}\label{prop:generalized-doubling}
Suppose that $S$ is a connected link cobordism from $K_0$ to $K_1$ with
$M$ local maxima. Suppose that $s$, $t$, $p$, and $q$ are
non-negative integers such that
\[
s+t = M,\qquad  p+q = M + 2g(S), \qquad s \le p, \quad \text{and} \quad t\le q.
\]
Then there is a decoration $\cF$ of $S$, as well as a decoration $\cF'$
of the mirrored cobordism $\bar{S}$, such that
\[
u^{s}v^{t} \cdot F_{\cF'} \circ F_{\cF} \simeq u^{p}v^{q} \cdot \id_{\cCFK^-(K_0)}.
\]
\end{prop}

\begin{proof}
The proof follows from the same strategy as the proof of
Proposition~\ref{prop:main-thm-maps}, with some extra care taken regarding
the dividing set.
By a sequence of band slides, we can ensure that there are disjoint and
connected subarcs $a_1$, $a_2 \subset K_0$, with respect to which $S$ has
the following movie:
\begin{enumerate}[leftmargin=2cm, label=($M$-\arabic*), ref=$M$-\arabic*]
\item\label{movie:1'} $m$ births, each adding an unknot.
\item\label{movie:2'} $m$ fusion bands, each connecting an unknot to $K_0$.
\item\label{movie:3'} $2g(S)$ bands, with attaching feet in $a_1$.
Furthermore, these bands come in pairs which have linked attaching feet along $K_0$.
\item\label{movie:4'} $M$ fission bands, with ends in $a_2$. Both feet of each band are adjacent on $K_0$.
\item\label{movie:death} $M$ deaths, each removing an unknot.
\item\label{move:5'} An isotopy, moving the band surgered copy of $K_0$ to $K_1$.
\end{enumerate}

\begin{figure}[ht!]
\centering
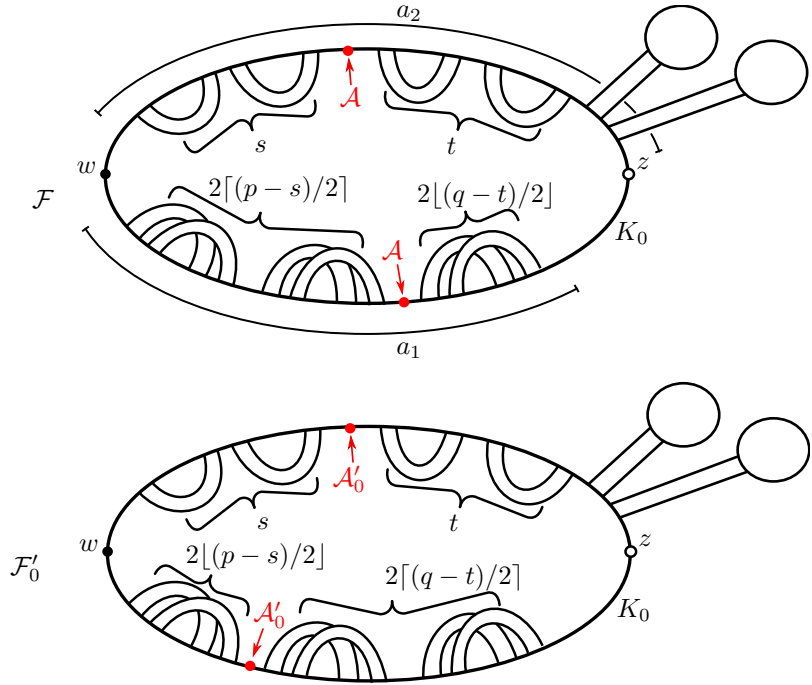
\caption{The configuration of the bands of $S$, attached to $K_0$.
The dividing sets $\cA$ on $S$ and $\cA_0'$ on $\bar{S}$ are given by the red dots
(extended horizontally) in the top and bottom figure, giving rise to the decorated
surfaces $\cF$ and $\cF_0'$, respectively.}\label{fig:8}
\end{figure}

Since
\[
(p-s) + (q-t) = 2g(S),
\]
we conclude that $p-s$ and $q-t$ have the same parity. Consequently,
\[
\left \lceil \frac{p-s}{2}\right\rceil+\left \lfloor\frac{q-t}{2}\right\rfloor=
\left \lfloor \frac{p-s}{2}\right\rfloor+\left \lceil \frac{q-t}{2}\right\rceil=g(S).
\]

Construct a dividing set $\cA$ on $S$ with 2 arcs such that the $\ws$ and $\zs$-subregions are connected, and
\begin{enumerate}
\item $s$ of the fission bands from step~\eqref{movie:4'} occur in the $\ws$-subregion,
and the other $t$ bands occur in the $\zs$-subregion.
\item $2\lceil (p-s)/2\rceil$ linked bands (from  the $\lceil (p-s)/2\rceil $ pairs of linked bands)
from step~\eqref{movie:3'} occur in the $\ws$-subregion,
while the other $2\lfloor(q-t)/2\rfloor$ occur in the $\zs$-subregion.
\end{enumerate}

We now construct a decoration on the turned around cobordism $\bar{S}$.
We first construct a decoration $\cA_0'$, which does not quite match up with
the decoration on $S$ along $K_1$, and gives rise to the decorated surface $\cF_0' = (\bar{S}, \cA_0')$.
We construct $\cA_0'$ such that the following hold:
\begin{enumerate}
\item $s$ of the fission bands from step~\eqref{movie:4'} occur in the $\ws$-subregion,
and the other $t$ bands occur in the $\zs$-subregion.
\item $2\lfloor (p-s)/2\rfloor$ linked bands (from  the $\lfloor (p-s)/2\rfloor$ pairs of linked bands)
from step~\eqref{movie:3'} occur in the $\ws$-subregion, while the other $2\lceil(q-t)/2\rceil$ occur in the $\zs$-subregion.
\end{enumerate}

The dividing arc of $\cA_0'$ which separates the fission bands can always be
chosen to match with a dividing arc of $\cA$ (this corresponds to the top red
dot of Figure~\ref{fig:8}). Our description of the other two arcs do not
match up along $K_1$. Nonetheless, we can construct a decoration $\tilde{\cA}$
on $[0,1] \times K_1$, consisting of two arcs that do not cross
$[0,1] \times \{w\}$ or $[0,1] \times \{z\}$, which connect the endpoints of the dividing
sets of $\cA_0'$ and $\cA$. We define the decoration  on $\cF'$ to be the
union of $\cA_0'$ and $\tilde{\cA}$.

We delete the deaths of step~\eqref{movie:death} from $\cF$, and also delete the
corresponding births from $\cF'$. We glue the resulting boundary components together in pairs via horizontal cylinders. The resulting surface is obtained by adding
$s$ tubes to the $\ws$-subregion, and $t$ tubes to the $\zs$-subregion. Let $\cG$
denote the resulting decorated surface. A generalization of
Lemma~\ref{lem:stabilization} implies that adding a tube to the
$\zs$-subregion changes the link cobordism map by multiplication by $v$, and
adding a tube to the $\ws$-subregion changes the map by multiplication by
$u$; see \cite{JZStabilization}*{Lemma~5.3} for a proof. Consequently,
\begin{equation}
F_{\cG}=u^{s}v^{t}\cdot F_{\cF'}\circ F_{\cF}.\label{eq:G=usvtF'F}
\end{equation}

The surface $\cG$ has $p+q$ distinguished tubes (one tube for each  band attached to $K_0$ to form $\cF$).
 Let $\cG_0$ denote the decorated link cobordism obtained by removing these tubes from $\cG$,
and decorating the resulting surface with a horizontal pair of dividing arcs.

We claim that
\begin{equation}
F_{\cG}=u^p v^q\cdot F_{\cG_0}.\label{eq:G=upvqG0}
\end{equation}
First, note that, by construction, $s+2\lfloor(p-s)/2\rfloor$ of the tubes
 occur fully in the $\ws$-subregion, and $t+2\lfloor (q-t)/2\rfloor$
  tubes occur fully in the $\zs$-subregion. If $p-s$ and $q-t$ are both even, then
Equation~\eqref{eq:G=upvqG0} follows from Lemma~\ref{lem:stabilization}. If $p-s$ and
$q-t$ are both odd, then there are exactly two tubes which are not fully in the $\ws$-subregion,
or in the $\zs$-subregion. Using Lemma~\ref{lem:stabilization} to remove the
$p+q-2$ tubes which are fully in the $\ws$-subregion or the $\zs$-subregion, it
remains to show Equation~\eqref{eq:G=upvqG0} when $p=q=1$ and $s=t=0$. The dividing set
of $\cG$ is shown in Figure~\ref{fig:9}.

\begin{figure}[ht!]
\centering
\begingroup%
  \makeatletter%
  \providecommand\color[2][]{%
    \errmessage{(Inkscape) Color is used for the text in Inkscape, but the package 'color.sty' is not loaded}%
    \renewcommand\color[2][]{}%
  }%
  \providecommand\transparent[1]{%
    \errmessage{(Inkscape) Transparency is used (non-zero) for the text in Inkscape, but the package 'transparent.sty' is not loaded}%
    \renewcommand\transparent[1]{}%
  }%
  \providecommand\rotatebox[2]{#2}%
  \newcommand*\fsize{\dimexpr\f@size pt\relax}%
  \newcommand*\lineheight[1]{\fontsize{\fsize}{#1\fsize}\selectfont}%
  \ifx\svgwidth\undefined%
    \setlength{\unitlength}{331.06099371bp}%
    \ifx\svgscale\undefined%
      \relax%
    \else%
      \setlength{\unitlength}{\unitlength * \real{\svgscale}}%
    \fi%
  \else%
    \setlength{\unitlength}{\svgwidth}%
  \fi%
  \global\let\svgwidth\undefined%
  \global\let\svgscale\undefined%
  \makeatother%
  \begin{picture}(1,0.36119679)%
    \lineheight{1}%
    \setlength\tabcolsep{0pt}%
    \put(0,0){\includegraphics[width=\unitlength,page=1]{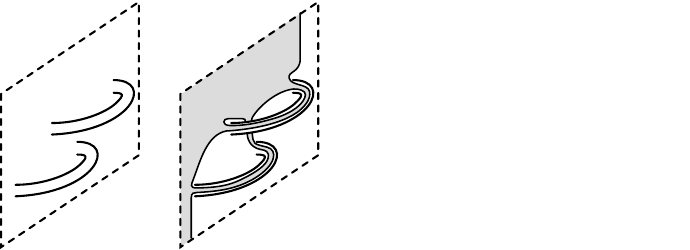}}%
    \put(0.32062657,0.01151029){\color[rgb]{0,0,0}\makebox(0,0)[lt]{\lineheight{1.25}\smash{\begin{tabular}[t]{l}$\cG$\end{tabular}}}}%
    \put(0,0){\includegraphics[width=\unitlength,page=2]{fig9.pdf}}%
    \put(0.58064604,0.01151016){\color[rgb]{0,0,0}\makebox(0,0)[lt]{\lineheight{1.25}\smash{\begin{tabular}[t]{l}$\cG_0$\end{tabular}}}}%
    \put(0,0){\includegraphics[width=\unitlength,page=3]{fig9.pdf}}%
    \put(0.93562365,0.05565665){\color[rgb]{0,0,0}\makebox(0,0)[lt]{\lineheight{1.25}\smash{\begin{tabular}[t]{l}$\hat{D}$\end{tabular}}}}%
    \put(0,0){\includegraphics[width=\unitlength,page=4]{fig9.pdf}}%
  \end{picture}%
\endgroup%

\caption{Far-left: the underlying surface of $\cG$.
Middle-left: the decoration on $\cG$. Middle-right: the destabilized $\cG_0$. Far-right: the disk $\hat{D}\subset \cG_0$.}\label{fig:9}
\end{figure}

Let $\hat{D}\subset \cG_0$ denote a disk which contains the 4 feet of the two tubes,
and also intersects the dividing set of $\cG_0$ in a single arc.
We may pick $\hat{D}$ to consist of the product
of a subarc of $K_0$, containing the 4 feet of the tubes, and a sub-interval of $[0,1]$.
Let $\gamma$ be a path in $\hat{D}$ connecting a foot of one tube to
a foot of the other tube. Viewing $K_1$ as the middle level
set of the doubled surface, we assume $\g$ is chosen to be a subarc of $K_1$,
which is disjoint from the bands. Let $h_1$ and $h_2$ be 3-dimensional 1-handles, corresponding
to the two tubes. Let $B\subset [0,1]\times S^3$ denote a regular neighborhood
of $h_1\cup h_2\cup \gamma$. We note $B$ is topologically a 4-ball.

The surfaces $\cG_0$ and $\cG$ intersect $\d B$ in a
3-component unlink. This can be seen as follows.
We let $B_0\subset S^3$ denote a 3-ball which
contains the two bands corresponding to
$h_1$ and $h_2$, as well as a sub-arc of $K_1$ corresponding to $\g$.
We may take $B$ to be $[a,b]\times B_0$, for
some subinterval $[a,b]\subset [0,1]$, where the two bands
and their mirrors are attached in the time interval $[a,b]$.
The boundary of $B$ consists of the union of $\{a,b\}\times B_0$
and $[a,b]\times \d B_0$. By construction,
 $\cG\cap \d B=\cG_0\cap \d B$. Furthermore, we claim that $\cG\cap \d B$
 is a 3-component unlink.
To see this, we note $\cG\cap \d B$ consists of the union
of $\{a,b\}\times (B_0\cap K)$ and $[a,b]\times \d (B_0 \cap K)$.
Since $B_0\cap K$ is a 3-component, boundary parallel tangle,
it follows that the link $\cG\cap \d B$ is a 3-component unlink.

In the language of
\cite{JZStabilization}*{Definition~2.8}, the underlying surface of $\cG$ is obtained by a
$(3,0)$-stabilization of $\cG_0$.

The dividing set of $\cG_0$ intersects $\hat{D}$ in a single arc.
The dividing set of $\cG$ intersects the union of $\hat{D} \cap \cG$ and
the two tubes in a single arc; see the second frame of Figure~\ref{fig:9}.
 By \cite{JZStabilization}*{Lemma~5.3}, since the genera of the $\ws$- and $\zs$-subregions
are both one larger in $\cG$ than in $\cG_0$, it follows that
 \[
F_{\cG}=u^1v^1\cdot F_{\cG_0},
 \]
completing the proof of Equation~\eqref{eq:G=upvqG0} in the final case.

Note that $\cG_0$ is obtained by tubing in $m$ 2-spheres
into the identity concordance $[0,1]\times K_0$,
decorated with a horizontal pair of arcs. The proof of \cite{ZemRibbon}*{Theorem~1.7}
implies that tubing in 2-spheres in this manner does not change the cobordism maps,
so
\begin{equation}
F_{\cG_0}\simeq \id_{\cCFK^-(K_0)}.\label{eq:FG_0simeqid}
\end{equation}
Combining Equations~\eqref{eq:G=usvtF'F}, \eqref{eq:G=upvqG0}, and~\eqref{eq:FG_0simeqid}, we obtain
\[
u^tv^s \cdot F_{\cF'} \circ F_{\cF} \simeq u^pv^q \cdot \id_{\cCFK^-(K_0)}.
\]
completing the proof.
\end{proof}

\subsection{Generalized torsions and knot cobordisms}
We now state a generalization of Theorem~\ref{thm:cobordism-bound} involving the chain torsion order:

\begin{prop}\label{prop:bound-generalized-torsion}
Suppose there is a connected knot cobordism $S$ from $K_0$ to $K_1$ with $M$ local maxima. Then
\[
\cOrd_{u,v}^{\Chain}(K_0) \le \max\{M, \cOrd_{u,v}^{\Chain}(K_1)\} + 2g(S).
\]
\end{prop}

\begin{proof}
Suppose that $i$, $j$ are non-negative integers such that
\begin{equation}\label{eq:main-assumption-i-j}
\max\{M, \cOrd_{u,v}^{\Chain}(K_1)\} + 2g(S) \le i + j.
\end{equation}
We claim that we can pick non-negative integers $s$, $t$, $p$, and $q$ such that
\begin{equation}
\begin{split}
&s \le p \le i,\\
&t \le q \le j,\\
&s + t = M,\\
&p + q = M + 2g(S).
\end{split}
\label{eq:pqst-assumptions}
\end{equation}
Indeed, start by picking $s$ and $t$ such that $0 \le s\le i$, $0 \le t \le j$,
and $s + t = M$, which can be done since $M \le i+j$. Next, pick $p$ and $q$
such that $s \le p \le i$, $t\le q\le j$, and $p + q = M + 2g(S)$, which is possible since
$s$ and $t$ are already chosen, and $M + 2g(S) \le i+j$.

Consider the non-negative integers
\[
l_1:=i-p\qquad \text{and} \qquad l_2:=j-q.
\]
Equations~\eqref{eq:main-assumption-i-j} and~\eqref{eq:pqst-assumptions} imply that
\begin{equation}
\begin{split}
s+t+l_1+l_2&=s+t+i+j-p-q\\
&=i+j - 2g(S)\\
&\ge \cOrd_{u,v}^{\Chain}(K_1).
\end{split}
\label{eq:stl1l2>OK1}
\end{equation}

The generalized doubling relation from Proposition~\ref{prop:generalized-doubling}
implies that there are decorations $\cF$ of $S$ and $\cF'$ of the mirror $\bar{S}$, such that
\[
u^s v^t \cdot F_{\cF'} \circ F_{\cF}\simeq u^p v^q \cdot \id_{\cCFK^-(K_0)}.
\]
Multiplying by $u^{l_1}v^{l_2}$, we obtain
\begin{equation}\label{eq:f'g'=uiviid}
F_{\cF'}\circ(u^{s+l_1}v^{t+l_2}\cdot  F_{\cF})\simeq u^{i}v^j \cdot \id_{\cCFK^-(K_0)}.
\end{equation}
From equation~\eqref{eq:stl1l2>OK1}, we see that there are graded, $\bF_2[u,v]$-equivariant chain maps
\[
f\colon \cCFK^-(K_1)\to \bF_2[u,v] \qquad \text{and}\qquad g\colon \bF_2[u,v]\to \cCFK^-(K_1),
\]
such that $f\circ g$ and $g\circ f$ are both chain homotopic to multiplication by $u^{s+l_1}v^{t+l_2}$.

Set $g'=F_{\cF'} \circ g$ and $f' = f\circ F_{\cF}$.
Equation~\eqref{eq:f'g'=uiviid} implies that $g' \circ f'$ is chain homotopic
to $u^i v^j \cdot \id_{\cCFK^-(K_0)}$. The fact that $f' \circ g' \simeq
u^i v^j \cdot \id_{\bF_2[u,v]}$ follows since there is exactly one non-zero map
in $\Hom_{\bF_2[u,v]}(\bF_2[u,v],\bF_2[u,v])$ in each grading.
\end{proof}

\subsection{Topological bounds from the generalized torsion orders}

Many of the topological bounds we proved for $\Ord_v(K)$ also hold for the more general torsion orders:

\begin{prop}\label{prop:chain-torsion-applications}
Suppose $K$ is a knot in $S^3$.
\begin{enumerate}
\item Then $\cOrd_{u,v}^{\Chain}(K)\le ul_b(K)$, where $ul_b(K)$ is the band-unlinking number.
\item If $K$ is a ribbon knot, then $\cOrd_{u,v}^{\Chain}(K)\le \Fus(K)$.
\end{enumerate}
\end{prop}

\begin{proof}
The proofs are the same as the proofs of Corollary~\ref{thm:band-unknotting-number} and~\ref{cor:fusion-number-bound},
using Proposition~\ref{prop:bound-generalized-torsion} instead of Theorem~\ref{thm:cobordism-bound}.
\end{proof}

The most notable result which does not hold for $\cOrd_{u,v}^{\Chain}(K)$ is
our bound on the bridge index, Corollary~\ref{cor:bridge-number}. The proof
of Corollary~\ref{cor:bridge-number} used the fact that $\bF_2[u]$ is a PID,
which is not true for the ring $\bF_2[u,v]$.
Proposition~\ref{prop:chain-torsion-applications} instead implies that
\[
\cOrd_{u,v}^{\Chain}(K\# \bar{K})\le \br(K)-1.
\]
In the subsequent Section~\ref{sec:computations-generalized-torsion}, we will
compute several examples to illustrate the behavior of generalized torsion
orders.

\subsection{Computations of generalized torsion orders}
\label{sec:computations-generalized-torsion}

\begin{lem}\label{lem:generalized-torsion-order-computation} Suppose $p$ and $q$ are coprime and non-negative.
\begin{enumerate}
\item \label{lem-tor-comp:1}If $K$ is a positive L-space knot (e.g., $K=T_{p,q}$), then $\cHFK^-(K)$ is torsion-free
(i.e., $\cOrd_{u,v}(K)=0$), but is not free unless $K$ is the unknot.
\item \label{lem:tor-comp:2}$\cOrd_{u,v}^{\Hom}(T_{p,q})=(p-1)(q-1)/2$.
\item\label{lem:tor-comp:3} $\cOrd_{v}(\bar{T}_{p,q})=(p-1)(q-1)/2$.
\item\label{lem:tor-comp:4} $\cOrd_{v}(T_{p,q}\# \bar{T}_{p,q})=
\cOrd_{u,v}^{\Chain}(T_{p,q}\# \bar{T}_{p,q}) = \min \{p,q\}-1.$
\end{enumerate}
\end{lem}

\begin{proof}
\emph{Part~\ref{lem-tor-comp:1}}: If $K$ is an L-space knot, then the
complex $\cCFK^-(K)$ can be determined using Ozsv\'{a}th and Szab\'{o}'s
computation of the knot Floer homology of L-space knots, which we summarized
in Lemma~\ref{lem:torsion-order-L-space-knot}. For each generator of
$\CFK^\infty(K)$ over $\bF_2[U,U^{-1}]$, there is a corresponding generator
of $\cCFK^-(K)$ over $\bF_2[u,v]$. For each arrow in $\CFK^\infty(K)$, there
is a corresponding arrow in $\cCFK^-(K)$, which is weighted by
$u^{\alpha}v^\beta$, where $\alpha$ denotes the horizontal change of the
arrow, and $\beta$ the vertical change. If $x_0, x_1,\dots, x_{2n-1}, x_{2n}$
denote the generators of $\cCFK^-(K)$, then the kernel of the differential is
exactly the span of $x_0,x_2,\dots, x_{2n-2},x_{2n}$ over $\bF_2[u,v]$. The
differential introduces the relations
\[
\{u^{d_{2i-1}}\cdot x_{2i-2}=v^{d_{2i}}\cdot x_{2i}: 1\le i\le 2n\}.
\]
It is straightforward to see from this description that $\cHFK^-(K)$ is
torsion-free (there is an injection of $\bF_2[u,v]$-modules into
$\bF_2[u,v]$). It follows from the above relations that, if $K$ is an L-space knot, $\cHFK^-(K)$ is
free if and only if the Alexander polynomial is 1, which implies $K$ is the
unknot, since $K$ is an L-space knot.
See Figure~\ref{fig:6} for an example.

\begin{figure}[ht!]
\centering
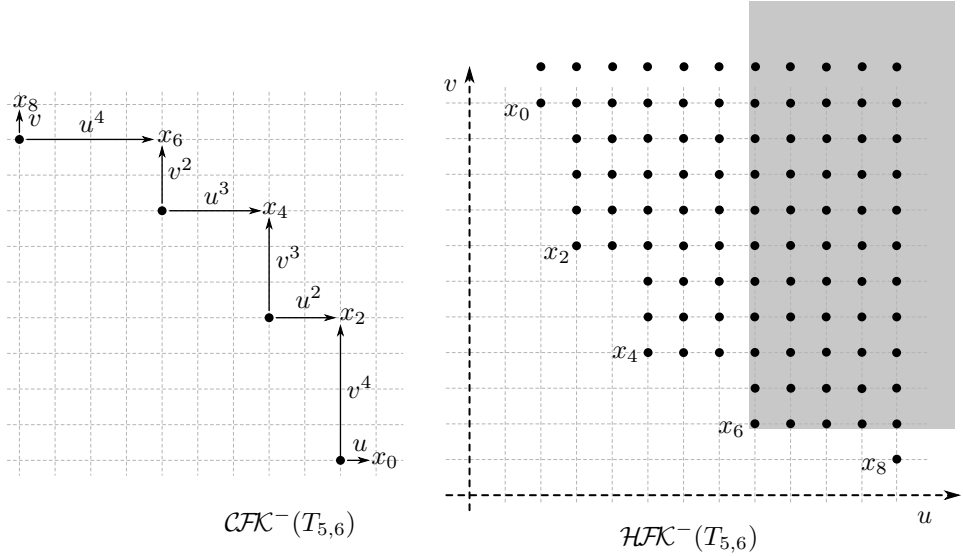
\caption{The complex $\cCFK^-(T_{5,6})$ and its homology $\cHFK^-(T_{5,6})$.
Each dot on the left denotes a generator over $\bF_2[u,v]$.
Each dot on the right denotes a generator over $\bF_2$.
The shaded rectangle is the span of $x_6$ over $\bF_2[u,v]$.}\label{fig:6}
\end{figure}

\emph{Part~\ref{lem:tor-comp:2}}: The homomorphism torsion order can be
rephrased as the minimum $N$ such that if $i$ and $j$ are non-negative
integers with $i+j=N$, then there is a rank 1, free submodule $F\subset
\cHFK^-(K)$ such that $u^iv^j\cdot \cHFK^-(K)\subset F$. For an L-space knot
$K$, the minimal such $N$ is easily seen to be $\frac{1}{2}\cdot
\sum_{i=1}^{2n} d_i$, where $d_i$ denotes the gaps in degrees of the
Alexander polynomial, as in Equation~\eqref{eq:gaps-alexander}. For L-space
knots, this is the Seifert genus of $K$. In particular, if $K=T_{p,q}$, we
obtain the stated formula.

\emph{Part~\ref{lem:tor-comp:3}}: The algebraic computation is performed in \cite{AEUnknotting}*{Example~5.1}.

\emph{Part~\ref{lem:tor-comp:4}}: By Equation~\eqref{eq:inequalities-of-torsion},
Proposition~\ref{prop:chain-torsion-applications}, and Corollary~\ref{cor:torsion-order-Tpq}, we have
\[
\cOrd_{v}(T_{p,q}\# \bar{T}_{p,q})\le \cOrd_{u,v}^{\Chain}(T_{p,q}\# \bar{T}_{p,q})\le
\Fus(T_{p,q}\# \bar{T}_{p,q}) = \min \{p,q\} - 1.
\]
Hence, it is sufficient to show that
\begin{equation}
\min\{p,q\}-1\le \cOrd_{v}(T_{p,q}\# \bar{T}_{p,q}).\label{eq:Ann-torsion-order-inequality-Tpq-Tpq}
\end{equation}
Assume $p<q$ for simplicity.

In Figure~\ref{fig:torusnewtorsion} (left and center), we draw portions of
$\cCFK^-(T_{p,q})$ and $\cCFK^-(\bar{T}_{p,q})$. Consider the element
\[
y:=x_2\otimes x_1'\in \cCFK^-(T_{p,q})\otimes_{\mathbb{F}_2[u,v]}
\cCFK^-(\bar{T}_{p,q})\iso\cCFK^-(T_{p,q}\#\bar{T}_{p,q}).
\]
Note that $\d y=0$.

An easy computation shows that
\[
\d(x_1\otimes x_1'+x_0\otimes x_0')=v^{p-1}\cdot x_2\otimes x_1',
\]
so $v^{p-1}\cdot [y]=0\in \cHFK^-(T_{p,q}\# \bar{T}_{p,q})$.

By Equation~\eqref{eq:CFK=cCFKtensor}, setting $u=0$ induces a chain map
\[
F\colon \cCFK^-(T_{p,q}\# \bar{T}_{p,q})\to \CFK^-(T_{p,q}\# \bar{T}_{p,q}),
\]
and hence an induced map on homology.

The complex $\CFK^-(T_{p,q}\#\bar{T}_{p,q})$ has a diamond shaped subcomplex
generated by $x_2\otimes x_1'$, $x_2\otimes x_2'$, $x_1\otimes x_2'$, and
$x_1\otimes x_1'$, as shown in Figure~\ref{fig:torusnewtorsion}. Moreover, no other differentials map to $x_2\otimes x_1'$.
Consequently, the element $F([y])\in \HFK^-(T_{p,q}\# \bar{T}_{p,q})$ has
$v$-torsion order $p-1$, and hence  $[y]$ must also have $v$-torsion order
$p-1$ in $\cHFK^-(T_{p,q}\# \bar{T}_{p,q})$.
Equation~\eqref{eq:Ann-torsion-order-inequality-Tpq-Tpq} follows, and hence
so does Claim~\ref{lem:tor-comp:4}.
\end{proof}

\begin{figure}[ht!]
\centering
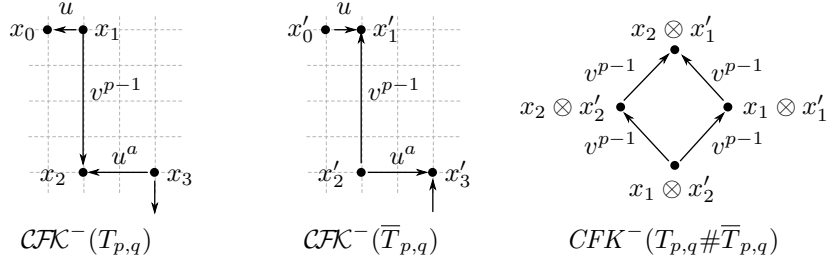
\caption{Portions of $\cCFK^-(T_{p,q})$ (left), $\cCFK^-(\bar{T}_{p,q} )$ (center),
and $\CFK^-(T_{p,q}\# \bar{T}_{p,q})$ (right), when $0<p<q$.}\label{fig:torusnewtorsion}
\end{figure}

Lemma~\ref{lem:generalized-torsion-order-computation} should be compared to the actual values
\[
ul_b(T_{p,q})=(p-1)(q-1) \quad \text{and} \quad \Fus(T_{p,q}\# \bar{T}_{p,q})=\br(T_{p,q})-1=\min \{p,q\}-1,
\]
which follow from Equation~\eqref{eq:relation-between-unlinking-number} and Corollary~\ref{cor:fusion-number-bound}.

\bibliographystyle{custom}
\bibliography{biblio}

%
%
\end{document}